\documentclass[12pt]{article}
\usepackage{amssymb,amsmath,amsthm, amsfonts}
\usepackage{graphicx}
\usepackage{epsfig}
\usepackage{mathrsfs}
\usepackage{amsthm}

\theoremstyle{plain}
\newtheorem{theorem}{Theorem}[section]
\newtheorem{lemma}{Lemma}[section]

\theoremstyle{definition}

\numberwithin{equation}{section}

\begin{document}

\title{Small-mass solutions in a two-dimensional logarithmic
chemotaxis-Navier-Stokes system with indirect nutrient  consumption}

\author{Ai Huang\\
 School of Mathematics and Statistics, Beijing Institute of Technology \\
  Beijing 100081, People's Republic of China\\
  Peter Y.~H.~Pang\\
Department of Mathematics, National University of Singapore\\
10 Lower Kent Ridge Road, Republic of Singapore 119076
\\
Yifu Wang\thanks{
Corresponding author. Email: {\tt wangyifu@bit.edu.cn}}\\
School of Mathematics and Statistics, Beijing
Institute of Technology\\
Beijing 100081, People's Republic of China}

\date{}
\maketitle
\vspace{0.3cm}
\noindent
\begin{abstract}
 This paper is concerned with the singular chemotaxis-fluid system with indirect nutrient consumption:
$
	%\left\{\begin{array}{l}
		n_{t}+u\cdot\nabla n=\Delta n-\nabla\cdot(n S(x,n,v)\cdot \nabla v);\ %\quad(x,t)\in\Omega\times(0,\infty),\\
		v_{t}+u\cdot\nabla v=\Delta v-vw;\
%\quad(x,t)\in\Omega\times(0,\infty),\\
		w_{t}+u\cdot\nabla w=\Delta w-w+n;\
%\quad(x,t)\in\Omega\times(0,\infty),\\
        u_t+(u\cdot\nabla) u=\Delta u-\nabla P+n\nabla\Phi;\
%\quad(x,t)\in\Omega\times(0,\infty)\\
\nabla\cdot u=0\
%	\end{array}
%	\right.
%\end{equation*}
$
 in a smooth bounded domain $\Omega\subset\mathbb{R}^2$
  under
no-flux/Neumann/Neumann/Dirichlet boundary conditions, where $\Phi\in W^{2,\infty}(\Omega)$, and
$S: \overline{\Omega}\times [0,\infty) \times (0,\infty)\rightarrow\mathbb{R}^{2\times 2}$
is a suitably smooth function that satisfies
$|S(x,n,v)|\leq
%\frac{
S_0(v)
%}{
/v
%}
$
%\qquad \hbox{
for all
%}\,\,
$(x,n,v) \in \Omega\times (0,\infty)^2$
with some nondecreasing $S_0: (0,\infty)\rightarrow(0,\infty)$.

For all reasonably regular initial data with a smallness assumption merely involving  the quantity   $\int_\Omega n_0$,
     it is shown that the problem possesses a  globally bounded  classical solution, which, inter alia,  exponentially stabilizes
    toward the spatially homogeneous state $( \frac{1}{|\Omega|}\int_{\Omega}n_0,0,\frac{1}{|\Omega|}\int_{\Omega}n_0,0)$ with respect to the norm in  $L^\infty(\Omega)$.
   This rigorously confirms that, at least in the two-dimensional setting,  in comparison to the direct mechanism of nutrient consumption, an indirect
mechanism can induce much more regularity of solutions to the chemotaxis--fluid system even with a singular tensor-valued sensitivity.
\end{abstract}

\vspace{0.3cm}


\newpage
\section{Introduction}

Chemotaxis,  the oriented movement of biological cells or organisms in response to an external chemical stimulation,
is known as a fundamental mechanism %of great significance
in a variety of biological processes such as pattern formation, wound healing, embryonic development and angiogenesis in tumor progression, and thus has attracted great interest at the level of theoretical understanding (\cite{BBTW,Budrene,Hillen(2009)}).

As  a prototypical version of nutrient-taxis models proposed in the second pioneering work of Keller and Segel (\cite{Keller(1971)}), the chemotaxis-consumption system
\begin{equation}\label{1.1}
\begin{cases}
n_t=\Delta n-\nabla\cdot(nS(x,n,v)\cdot\nabla v),\\
v_t=\Delta v-nv,
\end{cases}
\end{equation}
with the unknown functions   $n=n(x,t)$ and $v=v(x,t)$  representing the density of bacteria and   their nutrient concentration respectively,
has  been extensively studied over the past decades. A particular focus in this study is on the role of the absorptive term
$-nv$ in the $v$-equation toward the possible absence of finite-time mass accumulation. For  instance, when posed in bounded domains in $\mathbb{R}^d$ with arbitrary initial data $(n_0, v_0)$, the Neumann initial-boundary
 value problem of  \eqref{1.1}
 possesses globally bounded classical solutions in the case $d=2$ and global weak solutions when $d=3$ (\cite{Tao(2012)});
 and when $d\geq3$,  globally bounded classical solutions are warranted  whenever $\|v_0\|_{L^{\infty}(\Omega)}$ is suitably small (\cite{Tao{2011}}), but only global generalized solutions are constructed for large initial data (\cite{Winkler(2015)}).

In the prototypical case $S=id$,  \eqref{1.1} enjoys
 certain energy-like features due to  some favorable cancellation of the respective cross-diffusive contributions, which
  leads to a  energy-type inequality and thereby brings about the respective regularity of solutions. However in the case of general matrix-valued  sensitivities, a favorable quasi-energy structure of \eqref{1.1} seems no longer
present and the respective analysis thereof thus becomes significantly more complex (\cite{Winkler(2015),LiSuenWinkler(2015)}).
  %The above features however seem  to be inapplicable
 Furthermore,   % a considerable challenges arise
 if the consumptive chemotaxis model
 incorporates the effect that small changes in a stimulus affect the response of a biological agent more acutely
 at a low signal level
 than that in the presence of high signal concentrations
  (the so-called Weber-Fechner law of stimulus perception),  %in the way that
   then the sensitivity tensor $S(x,n,v)$ in
  \eqref{1.1}
   exhibits a singular behavior near $v=0$.  %(i.e. $S(x,n,v)=\frac1 v$),
   An early example of this phenomenon is the logarithmic chemotaxis system of the form
      \begin{equation}\label{1.2}
\begin{cases}
n_t=\Delta n-\nabla\cdot(\frac n v \nabla v),\\
v_t=\Delta v-nv,
\end{cases}
\end{equation}
 %obtained by letting $S(x,n,v)=\frac1 v id$,
 which exhibits the wave-like structured behavior (\cite{Keller(1971)}).
  Intuitively, the mathematical analysis of \eqref{1.2} is more challenging due to
the fact that  the absorptive
contribution to the second equations % in the $v$-equation of
induces the preference for small values of the attractive signal $v$,
 and accordingly % thereby considerably
 intensifies the destabilizing potential of the taxis interaction, %cross-diffusion
 % at small concentrates of  $v$,
resulting in quite complex forms of the structure formation of the biological agents (\cite{Li(2009),Li(2010),Jin(2013),Wang(2013)}).
As far as we know, results on global smooth solvability of \eqref{1.2} seem essentially limited to spatially one-dimensional
problems (\cite{Taowangwang(2013)}) or classes of suitably small initial data (\cite{WXY(2016),Winkler(2016)}). Even in two-dimensional settings
 global solutions can be constructed only within considerably generalized frameworks, but they
  %. Going beyond that, \eqref{1.2} possesses at least one global generalized solution in planer domains
%which
can  become eventually smooth and converge to the homogeneous steady state if the initial mass $\|n_0\|_{L^1(\Omega)}$
 is small
  (\cite{Winkler(20163M)}). In particular, under an explicit smallness condition on $n_{0}\ln n_{0}\in L^{1}(\Omega)$ and $\nabla\ln v_{0}\in L^{2}(\Omega)$, \eqref{1.2} possesses a global classical solution (\cite{Winkler(2016)}).
 Apart from that, in the radially symmetric setting,  a normalized solution of \eqref{1.2} was constructed
 in high-dimensional domains without any restriction
on the size of the initial data (\cite{Winkler(2018)}). As for the case of the sub-logarithmic sensitivity, that is $S(x,n,v)=\frac1 {v^\vartheta} id$
with $\vartheta\in (0,1)$, it  turns out that  nevertheless  \eqref{1.1}  admits  a global
 classical solution   under a smallness assumption on $\|v_0\|_{L^\infty(\Omega)}$ (\cite{Viglialoro}),
 and possesses globally  bounded solutions
  %for the sensitivity tensor $S(x,n,v)$
%satisfying  $|S(x,n,v)|\leq \frac {S_0} {v^\alpha} $,  %allowed to attain value in $\mathbb{R}^d\times \mathbb{R}^d $,
%the global boundedness of the corresponding  solutions  can be achieved
 whenever
  $\|n_0\|^{\frac 2d}_{L^1(\Omega)}(\|v_0\|_{L^\infty(\Omega)} +\|v_0\|^\theta_{L^\infty(\Omega)})$ with some $\theta>0$  is suitably small (\cite{WinklerDCDSB(2022)}).
 %, which is improved in the recent paper.
\vskip2mm

  Some experimental observation reveals that when the chemotactic motion of bacteria occurs in liquid environments,
 % the interaction of  the bacteria and surrounding should be taken into account in
bacterial aggregation can be substantially suppressed  by the fluid-driven transport, and vice versa  bacteria  potentially influence the
fluid flow through buoyant forces.  A  prototypical model for the above processes  is given by
\begin{equation}\label{1.3}
\begin{cases}
n_t+u\cdot\nabla n=\Delta n^m-\nabla\cdot(nS(x,n,v)\cdot\nabla v),\\
v_t+u\cdot\nabla v=\Delta v-nv,\\
u_t+ \kappa (u\cdot\nabla) u=\Delta u+n\nabla\Phi-\nabla P,\\
\nabla\cdot u=0,
\end{cases}
\end{equation}
 where the evolution of the fluid velocity $u$ is governed by a (Navier--)Stokes
equation ($\kappa\in \{0,1\}$) with pressure $P$ and gravitational potential
 $\Phi$ (\cite{Tuval(2005)}).
 In order to achieve an understanding of the mutual chemotaxis-fluid interaction,
 considerable efforts have been made to examine
  the balancing effect of porous medium-type diffusion  on the tendency toward cell
accumulation due to chemotactic cross-diffusion  in \eqref{1.3}. For example, the global existence  of bounded weak solutions
 to \eqref{1.3} with $\kappa=0$ and the tensor function $S\in C^2(\overline{\Omega} \times [0,\infty)^2; \mathbb{R}^{3\times 3})$ satisfying
\begin{equation}\label{1.4}
|S(x,n,v)|\leq
S_0(v)\,\, \hbox{for all}\, (x,n,v) \in \Omega\times [0,\infty)^2\, \hbox{with some nondecreasing }\,S_0:[0,\infty)\rightarrow \mathbb{R}
\end{equation}
is warranted when $m>\frac7 6$ (\cite{Winkler(cvpde2015)}), which is subsequently extended to $m>\frac {10}9 $ in \cite{ZhengKe(2021)} and $m>\frac{65}{63}$ in \cite{Tianxiang}, respectively.
In the linear borderline case $m = 1$, only global existence of  possibly
unbounded weak solutions to  \eqref{1.3} is available in the case of $d=3,\kappa=1$ and scalar-valued $S$ (\cite{Winkler Poincare(2016)});
or $d=2,\kappa=0$ and tensor-valued $S$ (\cite{Winkler(JEE2018)});
   moreover, such weak solutions enjoy some regularity features at least after suitably large waiting times
(\cite{WinklerTAMS(2017),Winkler(IMRN2021),Heihoff}). Apart from these,  there are some results on
particular versions of  \eqref{1.3} involving a singular chemotactic sensitivity. For instance,
for a possibly singular sensitivity $S$
  fulfilling
   \begin{equation}\label{1.5}
|S(x,n,v)|\leq \frac{S_0(v)}{v^{\frac 12}}\qquad \hbox{for all}\,\, (x,n,v) \in \Omega\times (0,\infty)^2
\end{equation}
   the no-flux-Dirichlet initial-boundary value problem
   of \eqref{1.3} with $\kappa=0,m\in (1,2]$  admits a globally bounded weak solution
in a smooth bounded convex domain $\Omega\subset\mathbb{R}^3$ (\cite{WinklerANS(2022)}).  In two-dimensional settings
 with the scalar sensitivity $S=\frac 1{v^\vartheta} id$ with $\vartheta\in (0,1)$,
the global classical  solvability of \eqref{1.3} is warranted for small initial data (\cite{LiuM3AS(2023)}); whereas for $S=\frac 1{v} id$,
the system \eqref{1.3} possesses a generalized solution  which is eventually smooth and
approaches the unique spatially homogeneous steady state in the large time limit if the initial data  is appropriately small (\cite{Black(2018),Liu(2021)}),
which  can be removed  %similar conclusion  can be warrant
when $-nv$ in the second equation of  \eqref{1.3} is replaced by $-nv^\theta$  with $\theta>1$ (\cite{FuestDCDSB(2023)}) or
the logistic-type kinetic  term is taken into account in the first equation of  \eqref{1.3} (\cite{Peterwangyin}).

In the classical Keller--Segel system,  chemical signals  partially directing the  movement of cells are self-secreted.
However, in some realistic biological circumstances
there are various  kinds of indirect-taxis mechanisms, such as the indirect prey-taxis in predator--prey ecosystems,
where the foraging  predator moves according to the gradient of some
chemical signal  released by the prey
 (see \cite{TelloWrzosek(2017),Strohm(2013),Surulescu(2021)}). In a considerable part of  the  literature it has shown that the  indirect-taxis mechanism can actually bring a substantial change in the rigorous analysis of the related models. (\cite{Fujie(2017),Daicvpde(2023),TaoWinkler(2017),Fuest(2019),Liu Li Huang(2020)}).

Accordingly, an interesting  natural  question is:
% problem is to clarify
%question is  %the present paper intends %attempts
%to address the question
 %to what extent
when the indirect-taxis mechanism is present, whether the singular chemotaxis-fluid system \eqref{1.3}  possesses
a global classical solution
 displaying a certain favorable decay behavior for all suitably regular initial data,
under a smallness condition  exclusively
involving the initial data through %total initial population mass
 the biologically relevant quantity $\int_{\Omega}n_{0}dx$ at least in the two-dimensional setting.
 %Specially, we shall consider the following
To this end, the present paper shall %intends to address the above question in the context of
consider the two-dimensional logarithmic
chemotaxis-Navier-Stokes system  of  the form
\begin{equation}\label{1.6}
\begin{cases}
		n_{t}+u\cdot\nabla n=\Delta n-\nabla\cdot(nS(x,n,v)\cdot\nabla v),
&\quad x\in\Omega,t>0,\\
		v_{t}+u\cdot\nabla v=\Delta v-vw,
&\quad x\in\Omega,t>0,\\
		w_{t}+u\cdot\nabla w=\Delta w-w+n,
&\quad x\in\Omega,t>0,\\
        u_t+(u\cdot \nabla)  u=\Delta u-\nabla P+n\nabla\Phi,
&\quad x\in\Omega,t>0,\\
\nabla\cdot u=0,
&\quad x\in\Omega,t>0,\\
(\nabla n-nS(x,n,v)\cdot\nabla v)\cdot \nu=\nabla v\cdot \nu=\nabla w\cdot \nu=0,u=0,
&\quad x\in\partial\Omega,t>0,\\
n(x,0)=n_0(x),v(x,0)=v_0(x),w(x,0)=w_0(x),u(x,0)=u_0(x),
&\quad x\in\Omega,
	\end{cases}
\end{equation}
where  the  nutrient  $v$ is consumed by the agent  $w$ secreted by the bacteria $n$, rather than by the bacteria directly, in a fluid environment.
Throughout the paper we  will assume that the given gravitational potential function $\Phi$ fulfills
$\Phi\in W^{2,\infty}(\Omega)$ and the chemotactic sensitivity tensor $S=(S_{ij})_{i,j\in \{1,2\}}$ satisfies
\begin{equation}\label{1.7a}
S_{ij}\in C^2(\overline \Omega \times [0,\infty)\times (0,\infty)) \quad \hbox{for}\, i,j\in \{1,2\}
\end{equation}
as well as
\begin{equation}\label{1.7b}
|S(x,n,v)|\leq \frac{S_0(v)}{v}\, \hbox{for all}\, (x,n,v) \in \Omega\times (0,\infty)^2
\hbox{ with some nondecreasing } S_0:[0,\infty)\rightarrow \mathbb{R}.
\end{equation}
As for  the initial data $(n_0,v_0,w_0,u_0)$ in \eqref{1.6}, we shall suppose that
\begin{equation}\label{1.8}
\left\{\begin{array}{l}
n_0\in C^\alpha(\bar{\Omega})\mbox{ for some }\alpha\in(0,1),\mbox{ with }n_0\geq0\mbox{ and }n_0\not\equiv0\mbox{ in }\bar{\Omega},\\
v_0\in W^{1,\infty}(\Omega)\mbox{ with }v_0>0\mbox{ in }\bar{\Omega},\\
w_0\in W^{1,\infty}(\Omega)\mbox{ with }w_0>0\mbox{ in }\bar{\Omega},\\
u_0\in D(A^\beta)\mbox{ for all }\beta\in(\frac1 2,1),
\end{array}
\right.
\end{equation}
where $A=-\mathcal{P}\Delta$ denotes the realization of the Stokes operator in $L^2(\Omega;\mathbb{R}^2)$ defined on the domain $D(A):=W^{2,2}(\Omega;\mathbb{R}^2)\cap W_0^{1,2}(\Omega;\mathbb{R}^2)\cap L_\sigma^{2}(\Omega;\mathbb{R}^2)$, with $L_\sigma^2(\Omega;\mathbb{R}^2):=\{\varphi\in L^{2}(\Omega;\mathbb{R}^2)|\nabla\cdot\varphi=0\}$ and $\mathcal{P}$ being the Helmholtz projection of $L^2(\Omega;\mathbb{R}^2)$ onto $L_\sigma^2(\Omega;\mathbb{R}^2)$, and for $\beta>0$,
$A^{\beta}$  denotes the corresponding sectorial fractional powers.

In this setting,  known results on \eqref {1.6} so far seem restricted to (\cite{Dai Liu(2023)}):
 for any $(n_0, v_0, w_0, u_0)$ satisfying \eqref{1.8},   \eqref {1.6} admits a  global classical solution %solvability
%to  system \eqref{1.6}  is derived
when  $S(x,n,v)=\frac 1v id$.  When $S(x,n,v)$ is not a scalar function but rather attains general values  in $\mathbb{R}^{2\times 2}$,  beyond the  global classical solvability  of \eqref{1.6},
 we intend  to establish the  global boundedness  of
   classical solutions, as well as their  exponentially  stabilization under a smallness assumption merely involving  the total mass
   $\int_\Omega n_0$. Our main result is stated as follows:
\begin{theorem}
	Let $\Omega\subset\mathbb{R}^2$ be a bounded domain with smooth boundary, and assume that $S$ satisfies
\eqref{1.7a}, \eqref{1.7b}. Then for all $K>0$ there is $\delta(K)>0$ with  the property that
 whenever the initial data $(n_{0},v_{0},w_{0},u_0)$ fulfills \eqref{1.8} as well as $\|v_{0}\|_{L^{\infty}(\Omega)}<K$ and
	\begin{equation}\label{1.10}
		\int_{\Omega}n_0< \delta(K),
	\end{equation}
 the system \eqref{1.6}  possesses a global classical solution which is bounded, and
\begin{equation}\label{1.11}
	n(\cdot,t)\to\overline{n_{0}}~\mbox{ in } L^{\infty}(\Omega),\quad v(\cdot,t)\to 0~\mbox{ in } L^{\infty}(\Omega),~~
w(\cdot,t)\to\overline{n_{0}}\mbox{ in } L^{\infty}(\Omega)\end{equation}
as well as
\begin{equation}\label{1.12}
\frac{\nabla v(\cdot,t)}{v(\cdot,t)}\to 0\mbox{ in } L^\infty(\Omega),\quad A^\beta u(\cdot,t)\to 0~\mbox{ in } L^{2}(\Omega)
\end{equation}
 exponentially as $t\to\infty$.
\end{theorem}
\vskip3mm
As compared to the case of  scalar sensitivities, a more delicate analysis  is required for tensor-valued ones
with regard to qualitative properties of solutions for the problem \eqref{1.6}.  In the present situation,
 the global boundedness of classical  solutions of  \eqref{1.6} is  based on  a priori information on
  approximate solutions $(n_\varepsilon, v_\varepsilon, w_\varepsilon,  u_\varepsilon, P_\varepsilon) $
  of the appropriately regularized problem \eqref{2.4}. The analysis of \eqref{2.4} relies on the evolution of the quantity
    \begin{equation*}
\mathcal{F}(n_\varepsilon,z_\varepsilon,w_\varepsilon,u_\varepsilon)(t)
:=\int_\Omega n_\varepsilon\ln\frac{ n_\varepsilon}{\overline{n_0}}
+\frac1 2\int_\Omega|\nabla z_\varepsilon|^2
+\int_\Omega(w_\varepsilon-\overline{w_\varepsilon})^2+\int_\Omega |u_\varepsilon|^2.
\end{equation*}
with $z_\varepsilon(x,t):=-\ln\frac{v_\varepsilon(x,t)}{\|v_{0}\|_{L^{\infty}(\Omega)}}$.
In fact, thanks to the  Moser--Trudinger inequality (Lemma \ref{lemma2.4}), $\mathcal{F}(n_\varepsilon,z_\varepsilon,w_\varepsilon,u_\varepsilon)(t)$
will become an energy functional after time $t_0^*>0$  under the smallness condition involving
$\int_\Omega n_0$ (Lemma \ref{lem3.6}).  Our strategy is to resort to a priori estimates for  the temporal average
$$
\frac1{t_2-t_1}\int_{t_1}^{t_2}
\int_\Omega (n\ln\frac n{\overline{n_0}}+
|\nabla z|^2+
(w-\overline{w})^2+
|u|^2)
$$
instead of  $
\frac1{t}\int_0^t
\int_\Omega (n\ln\frac n{\overline{n_0}}+
|\nabla z|^2+
(w-\overline{w})^2+
|u|^2),
$ which is different from  \cite{Black(2018),Liu(2021),Winkler(2020)}.
Thereupon, we use the resulting a priori bounds and decay properties of the approximate solutions
to establish regularity in a large time regime $t^*_2>0$ %independent of $\varepsilon$
through a bootstrap procedure (Lemma \ref{lem3.7}--Lemma \ref{lem3.10}).
%provided by the
It is the $\varepsilon$-independence of   $t^*_2>0$ that makes it possible to achieve a priori bounds for
some higher regularity  of the approximate solutions on the interval $(0, t^*_2)$.
Thereafter,  by way of compactness arguments,  this enables us to complete the first part of our
 main result that asserts
the global existence of  bounded classical solutions to \eqref{1.6},  along
with some decay properties (Lemma \ref{lemma3.11}). Furthermore,  on the basis of the
 decay properties of the global classical solution $(n,v,w,u)$, the exponential convergence properties asserted in the last part of Theorem 1.1
can be shown by employing a series of suitably organized bootstrap arguments.

The rest of this paper is organized as follows: In Section 2, we
 lay down some basic results and also recall some known properties of the Neumann heat semigroup and variants of the
  Moser--Trudinger inequality.
%, as well as some basic properties for \eqref{1.6}.
In Section 3, we employ a conditional energy functional $\mathcal{F}(t)$ to establish  a priori estimates
for approximate solutions  of \eqref{1.6}, and thereby obtain the global existence of  bounded classical solutions to \eqref{1.6}
along with favorable decay properties.  In
Section 4, the exponential convergence of classical solutions to \eqref{1.6}  is proved by means of  bootstrap arguments.

  %studies the asymptotic behavior of the global solutions for $n=2$.,
% when $n=2$ and the initial mass is small enough, we give a suitable upper bounded of the to

\section{Preliminaries}
In this section, we provide some preliminary results that will be used in the subsequent sections. We begin by recalling the important $L^p- L^q $ estimates for the  heat semigroup
under Neumann boundary conditions as given in \cite{Winkler(2010),CaoLankeit}.

\begin{lemma}\label{lemma2.1} %(Lemma 1.3 of \cite{Winkler(2010)})
	Let $(e^{t\Delta})_{t\geq0}$ be the Neumann heat semigroup in $\Omega$, and let $\lambda_{1}>0$ denote the first nonzero eigenvalue of $-\Delta$ in $\Omega$ under Neumann boundary conditions. Then there exist constants $c_{1},\dots,c_{4}$  depending upon $\Omega$  only which have the following properties.
\\
	(i) If $1\leq q\leq p\leq\infty$, then
	$$\|e^{t\Delta}\omega\|_{L^{p}(\Omega)}\leq c_{1}(1+t^{-\frac{n}{2}(\frac{1}{q}-\frac{1}{p})})e^{-\lambda_{1}t}\|\omega\|_{L^{q}(\Omega)}~\mbox{ for all } t>0$$
	holds for all $\omega\in L^{q}(\Omega)$ satisfying $\int_{\Omega}\omega=0.$\\
	(ii) If $1\leq q\leq p\leq\infty$, then
	$$\|\nabla e^{t\Delta}\omega\|_{L^{p}(\Omega)}\leq c_{2}(1+t^{-\frac{1}{2}-\frac{n}{2}(\frac{1}{q}-\frac{1}{p})})e^{-\lambda_{1}t}\|\omega\|_{L^{q}(\Omega)}~\mbox{ for all } t>0$$
	holds for each $\omega\in L^{q}(\Omega)$.\\
	(iii) If $2<q\leq p<\infty$, then
	$$\|\nabla e^{t\Delta}\omega\|_{L^{p}(\Omega)}\leq c_3(1+t^{-\frac{n}{2}(\frac{1}{q}-\frac{1}{p})})e^{-\lambda_{1}t}\| \nabla\omega\|_{L^{q}(\Omega)}~\mbox{ for all } t>0$$
	holds for each $\omega\in W^{1,q}(\Omega)$.\\
(iv) If $1<q\leq p\leq\infty$, then
	$$\|e^{t\Delta}\nabla\cdot\omega\|_{L^{p}(\Omega)}\leq c_{4}(1+t^{-\frac 12-\frac{n}{2}(\frac{1}{q}-\frac{1}{p})})e^{-\lambda_{1}t}\|\omega\|_{L^{q}(\Omega)}~\mbox{ for all } t>0$$
	holds for all $\omega\in (L^q(\Omega))^n$.
% $\omega\in (C_{0}^{\infty}(\Omega))^{n}$.\\
\end{lemma}

\begin{lemma}\label{lemma2.2}
	The Stokes operator $A$  is sectorial and generates an analytical semigroup semigroup $(e^{-tA})_{t\geq0}$ in
$L_\sigma^{p}(\Omega;\mathbb{R}^2)$. If we fix $0<\mu<\inf Re \sigma(A)$ where  $\sigma(A)$ denotes  the spectrum of $A$, then for any $p>1$ and $\beta>0$ there is $c(p,\beta)>0$ such that
	$$\|A^{\beta}e^{-tA}u\|_{L^{p}(\Omega)}\leq c(p,\beta)t^{-\beta}e^{-\mu t}\|u\|_{L^{p}(\Omega)}$$
	 for any $u\in L_\sigma^{p}(\Omega)$.
\end{lemma}
 The following two inequalities play an important role in the derivation of the associated  functional inequality and thereby  the global boundedness of solutions under consideration. Their proofs follow from the Moser--Trudinger inequality,
%The two inequalities that play an important rolntly establishing bounds e in subsequefor these solutions. For the proof of these two inequalities,
and we refer the interested reader to  \cite{Winkler(2020)} for the details.

\begin{lemma}\label{lemma2.4}
	Suppose that $\Omega \subset\mathbb{R}^{2}$ is a smooth bounded domain. Then
%there exists $\beta_0>0$ such that
for all $\varepsilon>0$
there exists $M=M (\varepsilon, \Omega)>0$ such that if $0 \not\equiv \phi \in C^{0}(\overline\Omega )$ is nonnegative and $\psi\in W^{1,2}(\Omega)$, then for each $a>0$,
%\begin{equation}
%		\begin{split}
%			\int_{\Omega}\phi | \psi | &\leq \frac{1}{a}\int_{\Omega}\phi \ln\frac{ \phi}{\overline{\phi}}+
%\frac{(1 + \varepsilon )a}{8\pi} \cdot \biggl\{\int_{\Omega}\phi \biggr\} \cdot\int_{\Omega} | \nabla \psi |^{2}\\
%			&\quad+M a\cdot \biggl\{\int_{\Omega} \phi \biggr\} \cdot \biggl\{ \int_{\Omega}| \psi | \biggr\}^{2}+
%\frac{M}{a}\int_{\Omega}\phi
\begin{equation}\label{2.1}
		\begin{split}
			\int_{\Omega}\phi | \psi | &\leq \frac{1}{a}\int_{\Omega}\phi \ln\frac{ \phi}{\overline{\phi}}+ \frac{(1 + \varepsilon )a}{8\pi} \cdot \biggl\{\int_{\Omega}\phi \biggr\} \cdot\int_{\Omega} | \nabla \psi |^{2}\\
			&\quad+M a\cdot \biggl\{\int_{\Omega} \phi \biggr\} \cdot \biggl\{ \int_{\Omega}| \psi | \biggr\}^{2}+\frac{M}{a}\int_{\Omega}\phi ,
		\end{split}
	\end{equation}
and
	\begin{equation}\label{2.2}
		\begin{split}		 \int_{\Omega}\psi\ln(\psi+1)&
\leq\frac{1+\varepsilon}{2\pi}
\biggl\{\int_{\Omega}\psi\biggr\}\cdot\int_{\Omega}\frac{|\nabla\psi|^{2}}{(\psi+1)^{2}}
\\			 &+4M\cdot\biggl\{\int_{\Omega}\psi\biggr\}^{3}+
\biggl\{M-\ln\big\{\frac{1}{|\Omega|}\int_{\Omega}\psi\big\}\biggr\}\cdot\int_{\Omega}\psi
		\end{split}
	\end{equation}
where $\overline{\phi}:= \frac{1}{|\Omega |}\int_{\Omega} \phi$.
\end{lemma}

The following elementary lemma provides some useful information on both the short-time
and the large-time behavior of certain integrals, which is used in the proof of Theorem 1.1.
\begin{lemma}\label{lemma2.5}
Let $\alpha<1$, $\beta<1$, and $\gamma$ and $\delta$ be positive constants such that $\gamma\neq\delta$. Then there exists $C>0$ such that
\begin{equation}
\int_0^t(1+(t-s)^{-\alpha})e^{-\gamma(t-s)}\cdot(1+s^{-\beta})e^{-\delta s}ds\leq C(1+t^{\min\{0,1-\alpha-\beta\}})e^{-\min\{\gamma,\delta\}t}
\end{equation}
for all $t>0$.
\end{lemma}

Due to the nonlinear no-flux boundary condition rather than the homogeneous Neumann boundary condition for $n$ in \eqref{1.6},
  it will be convenient to deal with  \eqref{1.6} by means of certain appropriate approximations.
  Following the approaches
pursued in \cite{Winkler(cvpde2015),LiSuenWinkler(2015)}, we fix a family of smooth functions
$\rho_\varepsilon\in C_0^\infty(\Omega)$, where $0\leq\rho_\varepsilon(x)\leq1$ for $\varepsilon\in(0,1)$ and
$\rho_\varepsilon(x)\nearrow 1$ in $\Omega$ as $\varepsilon\searrow 0$. For each $\varepsilon>0$, we define
$S _\varepsilon(x,n,v)=\rho_\varepsilon(x) S(x,n,v)$, and then consider the regularized variant of  \eqref{1.6} as follows:
\begin{equation}\label{2.4}
\begin{cases}
		n_{\varepsilon t}+u_\varepsilon\cdot\nabla n_\varepsilon=\Delta n_\varepsilon-\nabla\cdot(n_\varepsilon S_\varepsilon(x,n_\varepsilon,v_\varepsilon)\cdot\nabla v_\varepsilon),
&\quad x\in\Omega,t>0,\\
		v_{\varepsilon t}+u_\varepsilon\cdot\nabla v_\varepsilon=\Delta v_\varepsilon-v_\varepsilon w_\varepsilon,
&\quad x\in\Omega,t>0,\\
		w_{\varepsilon t}+u_\varepsilon\cdot\nabla w_\varepsilon=\Delta w_\varepsilon-w_\varepsilon+n_\varepsilon,
&\quad x\in\Omega,t>0,\\
        u_{\varepsilon t}+(u_\varepsilon\cdot \nabla)  u_\varepsilon=\Delta u_\varepsilon-\nabla P_\varepsilon+n_\varepsilon\nabla\Phi,
&\quad x\in\Omega,t>0,\\
\nabla\cdot u_\varepsilon=0,
&\quad x\in\Omega,t>0,\\
\nabla n_\varepsilon\cdot\nu=\nabla v_\varepsilon\cdot\nu=\nabla w_\varepsilon\cdot \nu=0,u_\varepsilon=0,
&\quad x\in\partial\Omega,t>0,\\
n_\varepsilon(x,0)=n_0(x),v_\varepsilon(x,0)=v_0(x),w_\varepsilon(x,0)=w_0(x),u_\varepsilon(x,0)=u_0(x),
&\quad x\in\Omega,
	\end{cases}
\end{equation}
which is globally solvable in the classical sense:
\begin{lemma}\label{lemma2.6}
For all $\varepsilon\in (0,1)$, there exist functions
 \begin{equation*}
		\left\{\begin{array}{l}
			n_\varepsilon\in C^{0}(\bar{\Omega}\times[0,\infty))\cap C^{2,1}(\bar{\Omega}\times(0,\infty)),\\
			v_\varepsilon\in C^{0}(\bar{\Omega}\times[0,\infty))\cap C^{2,1}(\bar{\Omega}\times(0,\infty)),\\
			w_\varepsilon\in C^{0}(\bar{\Omega}\times[0,\infty))\cap C^{2,1}(\bar{\Omega}\times(0,\infty)),\\
u_\varepsilon\in C^{0}(\bar{\Omega}\times[0,\infty); \mathbb{R}^2)\cap C^{2,1}(\bar{\Omega}\times(0,\infty); \mathbb{R}^2)\,\, \hbox{and}\\
P_\varepsilon\in C^{1,0}(\bar{\Omega}\times (0,\infty)),
		\end{array}
		\right.
	\end{equation*}
 such that $(n_\varepsilon,v_\varepsilon,w_\varepsilon,u_\varepsilon,P_\varepsilon)$  solves \eqref{2.4}
classically in $\Omega\times (0,\infty)$ and  $n_\varepsilon>0,v_\varepsilon>0,w_\varepsilon >0$ in $
 \bar{\Omega}\times (0,\infty)$. Furthermore,
\begin{equation}\label{2.5}
\int_\Omega  n_\varepsilon(\cdot,t)=\int_\Omega  n_0 \quad\hbox{for all}  \,\, t>0
\end{equation}
as well as \begin{equation} \label{2.6}
 \|v_\varepsilon(\cdot,t)\|_{L^\infty(\Omega)}\leq\|v_{0}\|_{L^{\infty}(\Omega)}
 \quad\hbox{for all}  \,\, t>0.
\end{equation}
\end{lemma}
\begin{proof}[Proof]
 From \eqref{1.7b}, it follows that  $|vS_\varepsilon(x,n,v)|\leq S_0(\|v_0\|_{L^\infty(\Omega)})$. Thereupon, by the arguments in
 %according to
 \cite [Theorem 1.1]{Dai Liu(2023)} with possible minor modifications, \eqref{2.4} admits a unique global classical solution
 $(n_\varepsilon, v_\varepsilon, w_\varepsilon,  u_\varepsilon, P_\varepsilon) $, up to the addition of constants
  to the pressure $P_\varepsilon$. The identity \eqref{2.5} directly results upon integration of the first equation in \eqref{2.4} with respect to $x\in\Omega$, whereas
  \eqref{2.6} is a straightforward consequence of the maximum principle applied to the second equation in \eqref{2.4}, due to $w_\varepsilon\geq 0$.
    \end{proof}

\section{Global Boundedness}
In this section, under the condition that $m:=\int_{\Omega}n_{0}$  is appropriately small, we shall establish the global  boundedness of the solutions to \eqref{1.6} in planer domains,
which is a consequence of $L^\infty(\Omega)$-estimates of solutions $(n_\varepsilon, v_\varepsilon, w_\varepsilon,  u_\varepsilon, P_\varepsilon) $
 obtained in Lemma \ref{lemma2.6}.
In order to derive the relevant estimates thereof,
we will first look at a de-singularized version of \eqref{2.4}.
%term of the form  $\nabla\cdotp (u \nabla w)$  instead  of $\nabla\cdotp (\frac {u}v\nabla v)$ in \eqref{1.1}. To this end, we employ the following transformation as  in \cite{Lankeitnon,Lankeitjde,Winkler3M}:
In view of the properties of the logarithmic function and \eqref{2.6}, we %the fact that $v(x,t)\leq\|v_{0}\|_{L^{\infty}(\Omega)}$, we
introduce  the following nonnegative function (see e.g. %as in
\cite{Winkler(2016),Black(2018)}):
\begin{equation}\label{3.1}
	z_\varepsilon(x,t):=-\ln\frac{v_\varepsilon(x,t)}{\|v_{0}\|_{L^{\infty}(\Omega)}}\mbox{ for all }(x,t)\in\Omega\times[0,T_{max}),\quad z_0(x)=-\ln\frac{v_0(x)}{\|v_{0}\|_{L^{\infty}(\Omega)}},
\end{equation}
and turn \eqref{2.4} into the non-singular chemotaxis system
\begin{equation}\label{3.2}
	\left\{\begin{array}{l}
		n_{\varepsilon t}+u_\varepsilon\cdot\nabla n_\varepsilon=\Delta n_\varepsilon+\nabla\cdot(n_\varepsilon v_\varepsilon S(x,n_\varepsilon,v_\varepsilon)\cdot \nabla z_\varepsilon),\quad x\in\Omega,t>0,\\
		z_{\varepsilon t}+u_\varepsilon\cdot\nabla z_\varepsilon=\Delta z_\varepsilon-|\nabla z_\varepsilon|^{2}+w_\varepsilon,\quad x\in\Omega,t>0,\\
		w_{\varepsilon t}+u_\varepsilon \cdot\nabla w_\varepsilon=\Delta w_\varepsilon-w_\varepsilon+n_\varepsilon,\quad x\in\Omega,t>0,\\
u_{\varepsilon t}+u_\varepsilon\cdot\nabla u_\varepsilon=\Delta u_\varepsilon-\nabla P_\varepsilon+n_\varepsilon \nabla\Phi,
\quad\nabla\cdot u_\varepsilon=0,\quad x\in\Omega,t>0,\\
		\frac{\partial n_\varepsilon}{\partial\nu}=\frac{\partial z_\varepsilon}{\partial\nu}=\frac{\partial w_\varepsilon}{\partial\nu}=0,\quad u_\varepsilon=0,\quad x\in\partial\Omega,t>0,\\ n_\varepsilon(\cdot,0)=u_{0},z_\varepsilon(\cdot,0)=z_{0},w_\varepsilon(\cdot,0)=w_{0},u_\varepsilon(\cdot,0)=u_0,\quad x\in\Omega.
	\end{array}\right.
\end{equation}

 In what follows, unless otherwise stated, without further explicit mentioning, we shall drop the index $\varepsilon$  of the global classical solution
    $(n_\varepsilon, z_\varepsilon, w_\varepsilon,  u_\varepsilon, P_\varepsilon) $ provided by Lemma \ref{lemma2.6} for simplicity.

Now we will derive some basic properties of the solutions to \eqref{3.2} as a starting point of the subsequent analysis.
\begin{lemma}\label{lemma3.1}
	Let $(n,z,w,u,P)$ be the global classical solution of \eqref{3.2} and $K>0$.
 %with  $z$ as given in \eqref{3.1}.
 If $\|v_{0}\|_{L^{\infty}(\Omega)}<K$,  then
\begin{equation}\label{3.3}
\int_\Omega w(\cdot,t)\leq m+e^{-t}\int_\Omega w_0 \,\,  \mbox{ for all} \ t\geq 0,
\end{equation}
and for any $0\leq t_1<t_2$, we have
\begin{equation}\label{3.4}
\frac1{t_2-t_1}\int_{t_1}^{t_2}\int_\Omega|\nabla z|^2\leq\frac{C_0(t_1+1)}{t_2-t_1}+m+e^{-t_1}\int_\Omega w_0,
\end{equation}
as well as
\begin{equation}\label{3.5}
\frac1{t_2-t_1}\int_{t_1}^{t_2}\int_\Omega\frac{|\nabla n|^{2}}{(n+1)^2}\leq\frac{C_1(t_1+1)}{t_2-t_1}+mS_0^2(K)
+S_0^2(K)e^{-t_1}\int_\Omega w_0,
\end{equation}
where $C_0:=
\int_\Omega \ln \frac{\|v_0\|_{L^\infty(\Omega)}}{v_0}+ m+\int_\Omega w_0$ and $C_1:=2m+S_0^2(K)C_0$.
\end{lemma}
\begin{proof}[Proof] Since integrating the third equation in \eqref{3.2} shows that
\begin{equation*}
\frac d{dt}\int_\Omega w
=- \int_\Omega w+\int_\Omega n =m-\int_\Omega w,
\end{equation*}
the inequality \eqref{3.3} results from an ODE argument immediately.

Upon the integration  of the second equation in \eqref{3.2} and recalling the solenoidality of $u$, we obtain
\begin{equation}\label{3.6}
\frac d{dt}\int_\Omega z+\int_\Omega|\nabla z|^2
=\int_\Omega w-\int_\Omega u\cdot\nabla z
\leq m+e^{-t}\int_\Omega w_0,
\end{equation}
as well as
\begin{equation}\label{3.7}
\int_\Omega z(\cdot,t)\leq \int_\Omega z(\cdot,0)+\int_\Omega w_0+m t.
\end{equation}
Furthermore, we get
\begin{equation*}
\begin{split}
\frac1{t_2-t_1}\int_{t_1}^{t_2}\int_{\Omega}|\nabla z(\cdot,s)|^{2}ds
&\leq\frac1{t_2-t_1}\int_\Omega z(\cdot,t_1)
+m+e^{-t_1}\int_{\Omega}w_0\\
&\leq \frac{C_0 (t_1+1) }{t_2-t_1}+m+e^{-t_1}\int_\Omega w_0
\end{split}
\end{equation*}
for all $0\leq t_1<t_2$,  which is the desired inequality \eqref{3.4}.

Now, multiplying the first equation in (\ref{3.2}) by $\frac1{n+1}$, integrating by parts over $\Omega$ and using the solenoidality of $u$,
we derive that
\begin{equation}\label{3.8}
\begin{split}
\frac d{dt}\int_\Omega\ln(n+1)
&=\int_\Omega\frac1{n+1}
(\Delta n+\nabla\cdot(nS(\cdot,n,v)v\cdot\nabla z)-u\cdot\nabla n)\\
&=\int_\Omega\frac{|\nabla n|^2}{(n+1)^2}
+\int_\Omega\frac {nv}{n+1}(S\cdot\nabla z)\cdot\frac{\nabla n}{n+1}\\
&\geq \frac1 2\int_\Omega\frac{|\nabla n|^2}{(n+1)^2}
-\frac{S_0^2(\|v_{0}\|_{L^{\infty}(\Omega)})}{2}\int_\Omega|\nabla z|^2.
\end{split}
\end{equation}
Integrating \eqref{3.8} in time and by the fact $0\leq\ln(n+1)\leq n$, we obtain
\begin{equation*}
\begin{split}
&\frac1{t_2-t_1}
\int_{t_1}^{t_2}\int_{\Omega}\frac{|\nabla n|^2}{(n+1)^2}\\
&\leq\frac{S_0^2(\|v_{0}\|_{L^{\infty}(\Omega)})}{t_2-t_1}\int_{t_1}^{t_2}\int_{\Omega}|\nabla z|^2
+\frac2{t_2-t_1}\int_{\Omega}\ln(n(\cdot,t_2)+1)
-\frac2{t_2-t_1}\int_{\Omega}\ln(n(\cdot,t_1)+1)\\
&\leq\frac {S_0^2(\|v_{0}\|_{L^{\infty}(\Omega)})}{t_2-t_1}
\int_{t_1}^{t_2}\int_{\Omega}|\nabla z|^2
+\frac {2m} {t_2-t_1}\\
&\leq\frac{S_0^2(K)C_0 (t_1+1) }{t_2-t_1}
+mS_0^2(K)+S_0^2(K) e^{-t_1}\int_\Omega w_0+\frac {2m}{t_2-t_1},
\end{split}
\end{equation*}
which readily yields  \eqref{3.5}.
\end{proof}
\begin{lemma}For all $0\leq t_1<t_2$, we have
\begin{equation}\label{3.9}
\frac1{t_2-t_1}\int_{t_1}^{t_2}
\int_{\Omega}n\ln\frac n{\overline{n_0}}
\leq \frac m{\pi}\frac{C_1(1+t_1)}{t_2-t_1}
+(4Mm^2+M)m
+
(\frac{m^2}{\pi}+\frac{m\int_\Omega w_0}{\pi}e^{-t_1}
)S_0^2(K)%(\|v_0\|_{L^\infty(\Omega)})
+2m\ln_{+}\frac{|\Omega|}{m},
\end{equation}
 where  $M:=M(1,\Omega)$ and $C_1$ are provided by Lemma \ref{lemma2.4} and Lemma \ref{lemma3.1}, respectively.
\end{lemma}
\begin{proof}[Proof]
Observing that
\begin{equation*}
	\begin{split}
		\int_\Omega n\ln\frac n{\overline{n_0}}
&=\int_\Omega n\ln n-\int_\Omega n\ln\overline{n_0}\\
&\leq\int_\Omega n\ln(n+1)-m\ln\frac m{|\Omega|},
	\end{split}
\end{equation*}
we apply \eqref{2.2} with $\varepsilon=1$ to the first integral on the right-hand side to obtain
\begin{equation}\label{3.10}
	\begin{split}
\int_\Omega n\ln\frac n{\overline{n_0}}
&\leq\frac1{\pi}
\biggl\{\int_\Omega n\biggr\}
\cdot\biggl\{\int_\Omega\frac{|\nabla n|^2}{(n+1)^2}\biggr\}
+4M\biggl\{\int_\Omega n\biggr\}^3\\	
&\quad+\biggl\{M-\ln\{\frac1{|\Omega|}
       \int_\Omega n\}\biggr\}
\cdot\int_\Omega n-m\ln\frac m{|\Omega|}\\
&=\frac m{\pi}\int_\Omega\frac{|\nabla n|^2}{(n+1)^2}
+4Mm^3
+\bigl\{M-\ln\frac m{|\Omega|}\bigr\}
\cdot m+m\ln\frac{|\Omega|}{m}\\
&=\frac m{\pi}\int_\Omega\frac{|\nabla n|^2}{(n+1)^2}
+4Mm^3+Mm+2m\ln\frac{|\Omega|}{m}.
	\end{split}
\end{equation}
Then, integrating (\ref{3.10}) on $(t_1,t_2)$ along with (\ref{3.5}) entails
\begin{equation*}
	\begin{split}
\frac1{t_2-t_1}\int_{t_1}^{t_2}
\int_\Omega n\ln\frac n{\overline{n_0}}
&\leq\frac m{\pi}\frac1{t_2-t_1}
\int_{t_1}^{t_2}\int_\Omega\frac{|\nabla n|^2}{(n+1)^2}
+4Mm^3+Mm+2m\ln\frac{|\Omega|}{m}\\
&\leq \frac m{\pi}\frac{C_1(1+t_1)}{t_2-t_1}+(\frac{m^2}{\pi}
+\frac m{\pi}e^{-t_1}\int_\Omega w_0)S_0^2(K)%(\|v_0\|_{L^\infty(\Omega)})
\\
&\quad+4Mm^3+Mm+2m\ln_+\frac{|\Omega|}{m},\\
	\end{split}
\end{equation*}
which readily shows (\ref{3.9}).
\end{proof}

The following result on the evolution of $\int_\Omega|u|^2$ will  play an important role in
establishing the global boundedness and asymptotic behavior of smooth solutions.
\begin{lemma}\label{lemma3.3}
There exists $C_2:=C_2(M)>0$ such that
\begin{equation}\label{3.11}
\frac d{dt}\int_\Omega|u|^2
+\int_\Omega|\nabla u|^2
\leq C_2 m\int_\Omega n\ln\frac n{\overline{n_0}}+C_2m^2
\end{equation}for all $t>0$, where $M:=M(1,\Omega)$ is provided by Lemma \ref{lemma2.4}.
\end{lemma}
\begin{proof}[Proof]
We invoke the Poincar\'{e} inequality to fix $C_p>0$ fulfilling
\begin{equation}\label{3.12}
\int_\Omega|\varphi|^2\leq C_p\int_\Omega|\nabla\varphi|^2\mbox{ for all }\varphi\in W_0^{1,2}(\Omega,\mathbb{R}^2).
\end{equation}
Moreover, using the Cauchy-Schwarz inequality, we have
\begin{equation}\label{3.13}
\bigg\{\int_\Omega|u|\bigg\}^2
\leq|\Omega|\int_\Omega|u|^2
\leq C_p|\Omega|\int_\Omega|\nabla u|^2
\mbox{ for all }t>0.
\end{equation}
Thanks to Lemma \ref{lemma2.4}, testing the third equation in \eqref{3.2} against $u$ shows that
\begin{equation*}
\begin{split}
&\frac1 2\frac d{dt}\int_\Omega|u|^2+\int_\Omega|\nabla u|^2\\
&=
\int_\Omega nu\cdot\nabla\Phi\\
&\leq
\|\nabla\Phi\|_{L^\infty}\int_\Omega n|u|\\
&\leq
\|\nabla\Phi\|_{L^\infty}\int_\Omega n|u_1|
+\|\nabla\Phi\|_{L^\infty}\int_\Omega n|u_2|\\
&\leq
\|\nabla\Phi\|_{L^\infty}
\bigg\{\frac2{a}\int_\Omega n\ln\frac n{\overline{n_0}}
+\frac{am}{2\pi}\int_\Omega|\nabla u|^2
+2Mam(\int_\Omega |u|)^2+\frac{2Mm}{a}\bigg\}\\
&\leq
\frac{2\|\nabla\Phi\|_{L^\infty}}{a}
\int_\Omega n\ln\frac n{\overline{n_0}}
+\bigg\{\frac{am\|\nabla\Phi\|_{L^\infty}}{2\pi}
+2MamC_p|\Omega|\|\nabla\Phi\|_{L^\infty}\bigg\}
\int_\Omega|\nabla u|^2
+\frac{2Mm\|\nabla\Phi\|_{L^\infty}}{a}.
\end{split}
\end{equation*}
Choosing $a:=\frac1{2m\|\nabla\Phi\|_{L^\infty}(1+2MC_p|\Omega|)}$ and
$C: =4(M+1)\|\nabla\Phi\|_{L^\infty}^2(1+2MC_p|\Omega|)$,
we have
\begin{equation}\label{3.14}
\frac d{dt}\int_\Omega|u|^2
+\int_\Omega|\nabla u|^2\leq 2Cm\int_\Omega n\ln\frac n{\overline{n_0}}+2Cm^2,
\end{equation}
and thus completes the proof.
\end{proof}
\begin{lemma}\label{lemma3.4}
There exists $C_3=C_3(m)>0$ such that for all $t>0$,
\begin{equation}\label{3.15}
\int_\Omega w^2(\cdot,t)%+\int_\Omega|u(\cdot,t)|^2
\leq C_3 (t+1).
\end{equation}
\end{lemma}
\begin{proof}[Proof]
Testing the third equation in \eqref{3.2} against $w$, we apply Lemma \ref{lemma2.4} with $\phi=n, \psi=w$ and $\varepsilon=1$
to see that
\begin{equation*}
\begin{split}
\frac1 2&\frac d{dt}\int_\Omega w^2
+\int_\Omega w^2+\int_\Omega|\nabla w|^2\\
&=\int_\Omega nw\\
&\leq\frac1 a\int_\Omega n\ln\frac n{\overline{n_0}}
+\frac a{4\pi}(\int_\Omega n)\int_\Omega|\nabla w|^2
+Ma(\int_\Omega n)(\int_\Omega w)^2
+\frac M a(\int_\Omega n)\\
&\leq \frac1 a\int_\Omega n\ln\frac n{\overline{n_0}}
+\frac{am}{4\pi}\int_\Omega|\nabla w|^2+Ma(\max\{m,\|w_0\|_{L^1(\Omega)}\})^{2}m
+\frac {Mm}{a}.
\end{split}
\end{equation*}
 Taking $a:=\frac{2\pi}{m}$ in the above inequality, we get
\begin{equation}\label{3.16}
\frac d{dt}\int_\Omega w^2
+\int_\Omega w^2+\int_\Omega|\nabla w|^2
\leq
\frac m{\pi}\int_\Omega n\ln\frac n{\overline{n_0}}+4M \pi  (\max\{m,\|w_0\|_{L^1(\Omega)}\})^2+\frac{Mm^2}{2\pi}.
\end{equation}
On the other hand, it follows  from \eqref{3.9} that there is $c_1(m)>0$ such that
\begin{equation}\label{3.17}
\int_{0}^t
\int_{\Omega}n\ln\frac n{\overline{n_0}}
\leq  c_1(m)(t+1).
\end{equation}
Combining \eqref{3.17} with \eqref{3.16}, we have
$$
\int_\Omega w^2(\cdot,t)\leq \|w_0\|_{L^2(\Omega)}^2 e^{-t}+4M \pi  (\max\{m,\|w_0\|_{L^1(\Omega)}\})^2+\frac{Mm^2}{2\pi}
+mc_1(m)(t+1)
$$
 and readily arrive at  \eqref{3.15} with a suitable constant $C_3(m)>0$.
\end{proof}

The derivation of the global boundedness and asymptotic behavior of \eqref{2.4}
relies on an analysis of a conditional Lyapunov functional of the form
\begin{equation}\label{3.18}
\mathcal{F}(n,z,w,u)(t)
:=\int_\Omega n\ln\frac n{\overline{n_0}}
+\frac1 2\int_\Omega|\nabla z|^2
+\int_\Omega(w-\overline{w})^2+\int_\Omega |u|^2.
\end{equation}
As an intermediate step, we will show that the functional $\mathcal{F}(n,z,w,u)(t)$ is non-increasing under an extra smallness hypothesis concerning each component of the solutions.
\begin{lemma}\label{lem3.5}
Let $(n,z,w,u,P)$ be the global classical solution of \eqref{3.2}. There exist constants $m_*>0$, $0<\delta<1$, $\kappa>0$ and $C>0$ such that if $m<m_*$ and
\begin{equation}\label{3.19}
\mathcal{F}(t_0)\leq\delta\mbox{ for some }t_0>0,
\end{equation}
then
\begin{equation}\label{3.20}
\mathcal{F}(t)\leq\delta e^{-\kappa(t-t_0)}
\end{equation}
and
\begin{equation}\label{3.21}
\int_{t_1}^\infty\int_\Omega
\frac{|\nabla n|^2}{n}
+\int_{t_1}^\infty\int_\Omega|\Delta z|^2
+\int_{t_1}^\infty\int_\Omega|\nabla w|^2
+\int_{t_1}^\infty\int_\Omega|\nabla u|^2
\leq Ce^{-\kappa(t_1-t_0)}
\end{equation} for all  $t_1>t_0$.
\end{lemma}
\begin{proof}[Proof]
Noting the equation $\nabla\cdot u=0$ and the condition \eqref{1.7b},  upon integration by parts, one can see that for all  $t>0$,
%\begin{equation}
		\begin{align}\label{3.22}
			&\frac{d}{dt}\int_\Omega n\ln\frac n{\overline{n_0}}\notag\\
&=\int_\Omega\ln n\{\Delta n+\nabla\cdot(nv S(\cdot,n,v)\cdot\nabla z)
-u\cdot\nabla n\}\notag\\
			&=-\int_\Omega\frac{|\nabla n|^2}{n}
-\int_\Omega(v \nabla n\cdot  S(\cdot,n,v))\cdot\nabla z\\
			&\leq -\frac12 \int_\Omega\frac{|\nabla n|^2}{n}
+ \frac{ S_0^2(\|v_{0}\|_{L^{\infty}(\Omega)})}2 \int_\Omega n|\nabla z|^2\notag\\
			&\leq- \frac12\int_\Omega\frac{|\nabla n|^2}{n}
+ \frac{m S_0^2(K)}2 \int_\Omega |\nabla z|^2
+
 \frac{ S_0^2(K)}4 (\int_\Omega (n-m)^2+\int_\Omega|\nabla z|^4).\notag
		\end{align}
	%\end{equation}
Testing the second equation of \eqref{3.2} by $-\Delta z$ and making use of $\nabla\cdot u=0$, we have
\begin{equation}\label{3.23}
	\begin{split}
	&	\frac1 2\frac d{dt}\int_\Omega|\nabla z|^2
+\int_\Omega|\Delta z|^2\\
=&\int_\Omega(|\nabla z|^2-w+u\cdot\nabla z)\Delta z\\
=& \int_\Omega|\nabla z|^2\Delta z
-\int_\Omega (w-\overline{w})\Delta z
-\int_\Omega  \nabla z \cdot (\nabla u\cdot \nabla z)\\
%&\leq\frac1 2\int_\Omega|\Delta z|^2+\int_\Omega|\nabla z|^4+\int_\Omega(w-\bar{w})^2+\frac1 2\int_\Omega|\nabla u|^2+\frac1 2\int_\Omega|\nabla z|^4\\
\leq &
\frac1 2\int_\Omega|\Delta z|^2
+\frac3 2\int_\Omega|\nabla z|^4
+\int_\Omega(w-\overline{w})^2+\frac1 2\int_\Omega|\nabla u|^2
\quad \mbox{ for all } t>0.
	\end{split}
\end{equation}
In addition, from the $w$-equation in \eqref{3.2}, it follows that
\begin{equation}\label{3.24}
	\begin{split}
		\frac d{dt}\int_\Omega (w-\overline{w})^2
+2\int_\Omega|\nabla w|^2
&=-2\int_\Omega(w-\overline{w})^2
+2\int_\Omega (n-m)(w-\overline{w})\\
&\leq\int_\Omega (n-m)^2
-\int_\Omega(w-\overline{w})^2.
	\end{split}
\end{equation}
In view of $\nabla\cdot u=0$ and applying the Poincar\'{e} inequality
 $\int_\Omega |u|^2\leq C_p\int_\Omega|\nabla u|^2$
%and the classical Csisz\'{a}r--Kullback inequality,
%and the two-dimensional Sobolev  inequality  $ \int_\Omega(n-m)^2\leq C_S \{\int_\Omega|\nabla n|\}^2$,
we get
\begin{equation}\label{3.25}
\begin{split}
\frac1 2\frac d{dt}\int_\Omega |u|^2+\int_\Omega |\nabla u|^2+\frac1{4C_p}\int_\Omega|u|^2
=&
\int_\Omega(n\nabla\Phi-(u\cdot\nabla) u)\cdot u
+\frac1{4C_p}\int_\Omega|u|^2\\
\leq
&\|\nabla\Phi\|_{L^\infty(\Omega)}
\int_\Omega|n-m||u|+\frac1{4C_p}\int_\Omega|u|^2
\\
\leq
&\frac1{2C_p}\int_\Omega|u|^2
+C_p\|\nabla\Phi\|_{L^\infty(\Omega)}^2\int_\Omega(n-m)^2
\\
\leq &\frac1{2}\int_\Omega|\nabla u|^2
+C_p\|\nabla\Phi\|_{L^\infty(\Omega)}^2
\int_\Omega(n-m)^2.
\end{split}
\end{equation}
The combination of  \eqref{3.22}--\eqref{3.25} yields
\begin{align}\label{3.26}
%\begin{split}
&\mathcal{F}'(t)+\kappa\mathcal{F}(t)
+\frac12 \int_\Omega|\Delta z|^2
+2\int_\Omega|\nabla w|^2+\frac1 2 \int_\Omega|\nabla u|^2 +\frac1{2C_p} \int_\Omega |u|^2
%+\int_\Omega(w-\overline{w})^2
+\frac12\int_\Omega\frac{|\nabla n|^2}{n}\notag\\
\leq
&
C_1 \int_\Omega (n-m)^2+
(\frac{ S_0^2(K)}4+\frac3 2)  \int_\Omega|\nabla z|^4
+\kappa\int_\Omega n\ln\frac n{\overline{n_0}}
+\frac{\kappa} 2\int_\Omega(w-\bar{w})^2 \\
&+ (\kappa+\frac{m S_0^2(K)}4 )\int_\Omega|\nabla z|^2
+\kappa\int_\Omega |u|^2\notag
%\end{split}
\end{align}
with $C_1:=2C_p\|\nabla\Phi\|_{L^\infty(\Omega)}^2+1+ \frac{ S_0^2(K)}4$.

Now let us recall several inequalities which will be employed in the sequel. The Poincar\'{e} inequality  provides  $C_p>0$  such that
\begin{equation}\label{3.27}
\int_{\Omega}|\zeta|^2\leq C_p\int_{\Omega}|\nabla\zeta|^2\quad\textrm{for any}~\zeta\in C^1(\overline\Omega)~\textrm{with}~\int_\Omega\zeta=0~~
\end{equation}
as well as
\begin{equation}\label{3.28}
\int_{\Omega}|\nabla\zeta|^2\leq C_p\int_{\Omega}|\Delta \zeta|^2\quad\textrm{for each}~\zeta\in C^2(\overline\Omega)~\textrm{with}~\frac{\partial\zeta}{\partial\nu}=0
~\textrm{on}~\partial\Omega.
\end{equation}\label{3.29}
From the Gagliardo--Nirenberg inequality, there exists $C_{GN}>0$ such that
\begin{equation}\label{3.30}
\int_{\Omega}|\nabla \zeta|^4\leq C_{GN}\left(\int_{\Omega}|\Delta \zeta|^2\right)\cdot\int_{\Omega}|\nabla \zeta|^2\quad\textrm{for all}~\zeta\in C^2(\overline\Omega)~\textrm{with}~\frac{\partial\zeta}{\partial\nu}=0~\textrm{on}~\partial\Omega.
\end{equation}
The embedding $W^{1,1}(\Omega)\hookrightarrow L^2(\Omega)$  ensures the existence of $C_S>0$ such that
\begin{equation}\label{3.31}
\int_{\Omega}(\zeta-\bar\zeta)^2\leq C_S\left(\int_{\Omega}|\nabla \zeta|\right)^2\leq C_S\left(\int_{\Omega} \zeta\right)\cdot\int_{\Omega}\frac{|\nabla \zeta|^2}{\zeta}\quad
\textrm{for any}~0<\zeta\in C^1(\overline\Omega),
\end{equation}
where the second inequality follows from the H\"{o}lder inequality.
In addition, the logarithmic Sobolev inequality yields $C_{LS}>0$ fulfilling
\begin{equation}
\int_{\Omega}\zeta\ln\frac{\zeta}{\overline\zeta}\leq C_{LS}\int_{\Omega}\frac{|\nabla \zeta|^2}{\zeta}\quad\textrm{for all}~0<\zeta\in C^1(\overline\Omega).
\end{equation}

Applying the above inequalities in  \eqref{3.26}, we arrive at
\begin{equation}\label{3.32}
\begin{split}
&\mathcal{F}'(t)+\kappa\mathcal{F}(t)
+\big(\frac12-(\frac{mS_0^2(K)}4+\kappa )C_p-(\frac{C_{GN}S_0^2(K)}4+\frac3 2)\|\nabla z\|_{L^2}^2 \big)\int_\Omega|\Delta z|^2
+\frac14\int_\Omega|\nabla u|^2
\\
\leq & (\frac \kappa 2 C_p-2) \int_\Omega|\nabla w|^2
+(C_1C_S m+\kappa C_{LS}-\frac 12)\int_\Omega\frac{|\nabla n|^2}{n}
+(\kappa-\frac1{2C_p})\int_\Omega|u|^2.
\end{split}
\end{equation}
At the point, we pick
$$\kappa:=\min\{\frac1{8C_{LS}},\frac 1{ 8 C_p}\},\quad m_*:=\min\{\frac1{8C_1C_S},\frac1{2S_0^2(K)C_p}\}.$$
With these choices, and taking into account the fact that the nonnegativity of $\int_\Omega\ln\frac n{\overline{n_0}}$ implies $2\mathcal{F}\geq\|\nabla z(\cdot,t)\|_{L^2(\Omega)}^2$, we get
\begin{equation}\label{3.33}
\begin{split}
\mathcal{F}'(t)+&\kappa\mathcal{F}(t)
+\frac1 4\int_\Omega\frac{|\nabla n|^2}{n}
+\big(\frac1 4-(\frac {C_{GN}S_0^2(K)}4+\frac 32)\mathcal{F}(t)\big)\int_\Omega|\Delta z|^2\\
&+\int_\Omega|\nabla w|^2
+\frac1 4\int_\Omega|\nabla u|^2
+\frac1{8C_p}\int_\Omega|u|^2
\leq0.
\end{split}
\end{equation}
Now it can be concluded that  whenever \eqref{3.19} is valid with
$$\delta:=\frac1{6(C_{GN}S_0^2(K)+6)},$$
one will have
\begin{equation}\label{3.34}
\begin{split}
\mathcal{F}'(t)+&\kappa\mathcal{F}(t)
+\frac1 4\int_\Omega\frac{|\nabla n|^2}{n}
+\frac1 8\int_\Omega|\Delta z|^2+\int_\Omega|\nabla w|^2
+\frac1 4\int_\Omega|\nabla u|^2\leq0
\end{split}
\end{equation}
 for all $t>t_0$. Indeed, introducing
\begin{equation*}
	T^*:=\sup\big\{T>t_{0}\big|\mathcal{F}(t)<
\frac1{5(C_{GN}S_0^2(K)+6)}~\mbox{for all}~  t\in[t_{0},T)\big\},
\end{equation*}
we actually have $T^*=\infty$. To achieve this, assuming that  $T^*<\infty$, we then  have $\mathcal{F}(T^*)=\frac1{5(C_{GN}S_0^2(K)+6)}$ due to the continuity of
$\mathcal{F}$ with respect to $t$ and the fact that $\mathcal{F}(t)<\frac1{5(C_{GN}S_0^2(K)+6)}$ for all $t\in (t_{0},T^*)$,
whereupon  the combination of \eqref{3.19} and \eqref{3.34} readily leads to
$\mathcal{F}(t)<\mathcal{F}(t_0)\leq \frac1{6(C_{GN}S_0^2(K)+6)}$   for all $t\in (t_{0},T^*)$, which is absurd due to
$\mathcal{F}(T^*)=\frac1{5(C_{GN}S_0^2(K)+6)}$.
Hence
 the inequality in \eqref{3.20} readily  results from an ODE comparison argument. Furthermore, for any $t_1>t_0$,
 we integrate  \eqref{3.34} over $(t_1,\infty)$ to obtain
 \begin{equation*}
 	\begin{split}
& \int^\infty_{t_1}
\int_{\Omega}
\biggl(\frac{|\nabla n|^{2}}{n}(\cdot,s)+
|\Delta z(\cdot,s)|^{2}+|\nabla w(\cdot,s)|^{2}
+|\nabla u(\cdot,s)|^{2}\biggr)ds \\
\leq &8
\mathcal{F}(t_1)
\leq  8\delta e^{-\kappa (t_1-t_0)}~~\mbox{ for all } t_1>t_0,
\end{split}	\end{equation*}
 which establishes \eqref{3.21} with $C:=8\delta$.
\end{proof}

Making use of the dependence on $m$ of the
estimates established in Lemma \ref{lemma3.1}--Lemma \ref{lemma3.3},
we shall now show that \eqref{3.19} holds upon a condition on $m$.
%we confirm that  a further restriction on $m$ can actually warrant the validity of \eqref{3.19}.

\begin{lemma}\label{lem3.6}
There exist $m_0\leq m_*$, $\kappa>0$ and $C>0$ such that for all $(n_0,v_0,w_0,u_0)$ satisfying \eqref{1.8} with $\int_\Omega n_0<m_0$, we can find $t_0^*=t_0^*(n_0,v_0,w_0,u_0)>0$ having the properties that
\begin{equation}\label{3.35}
		\int_\Omega n(\cdot,t)\ln\frac{n(\cdot,t)}{\overline{n_{0}}}\leq e^{-\kappa(t-t_0^*)},%\mbox{ for all } t>t^*,
	\end{equation}
\begin{equation}\label{3.36}
	\int_\Omega|\nabla z(\cdot,t)|^2\leq e^{-\kappa(t-t_0^*)},%\mbox{ for all } t>t^*,
\end{equation}
\begin{equation}\label{3.37}
	\int_{\Omega}|w(\cdot,t)-\overline{w}(t)|^2\leq e^{-\kappa(t-t_0^*)},%\mbox{ for all } t>t^*
\end{equation}
and
\begin{equation}\label{3.38}
	\int_{\Omega}|u(\cdot,t)|^2\leq e^{-\kappa(t-t_0^*)},%\mbox{ for all } t>t^*;
\end{equation}
as well as
\begin{equation}\label{3.39}
	\int_t^\infty\int_\Omega\frac{|\nabla n|^2}{n}\leq Ce^{-\kappa(t-t_0^*)},%\mbox{ for all } t>t^*,
\end{equation}
\begin{equation}\label{3.40}
		\int_{t}^{\infty}\int_{\Omega}|\Delta z|^{2}\leq Ce^{-\kappa(t-t_0^*)},%\mbox{ for all } t>t^*,
\end{equation}
\begin{equation}\label{3.41}
		\int_{t}^{\infty}\int_{\Omega}|\nabla w|^{2}\leq Ce^{-\kappa(t-t_0^*)},%\mbox{ for all } t>t^*,
\end{equation}
and
\begin{equation}\label{3.42}
		\int_{t}^{\infty}\int_{\Omega}|\nabla u|^{2}\leq Ce^{-\kappa(t-t_0^*)}%\mbox{ for all } t>t^*.
\end{equation}
for all $t>t_0^*$.
\end{lemma}
\begin{proof}[Proof]  Using the solenoidality of $u$, we  may test the third equation in (\ref{3.2}) against $w-\overline{w}$ to see that
\begin{equation}\label{3.43}
\begin{split}
\frac1 2 \frac d{dt}&\int_\Omega(w-\overline{w})^2
+\int_\Omega|\nabla w|^2
+\int_\Omega(w-\overline{w})^2
+\int_\Omega u\cdot\nabla(w-\overline{w})(w-\overline{w})\\
&=\int_\Omega(n-\overline{n})(w-\overline{w})\\
&\leq\int_\Omega n(w-\overline{w})+m\int_\Omega|w-\overline{w}|.
\end{split}
\end{equation}
According to Lemma \ref{lemma2.4} with $\psi:=w-\overline{w}$, $\phi:=n$, $M:=M(1,\Omega)$ and $a:=1$, we obtain
\begin{equation}\label{3.44}
\begin{split}
\int_\Omega n(w-\overline{w})
&\leq\int_\Omega n|w-\overline{w}|\\
&\leq\int_\Omega n\ln\frac n{\overline{n_0}}
+\frac m{4\pi}\int_\Omega\nabla w|^2
+Mm \cdot \biggl\{\int_\Omega|w-\overline{w}| \biggr\}^2 +Mm.\\
\end{split}
\end{equation}
Therefore  substituting (\ref{3.40}) into (\ref{3.39}) and by the Young inequality, we have
\begin{equation}\label{3.45}
\begin{split}
\frac1 2 \frac d{dt}
&\int_\Omega(w-\overline{w})^2
+(1-\frac m{4\pi})\int_\Omega|\nabla w|^2+
\frac3 4\int_\Omega(w-\overline{w})^2\\
&\leq\int_\Omega n\ln\frac n{\overline{n_0}}
+Mm\biggl\{\int_\Omega|w-\overline{w}| \biggr\}^2 +Mm+m^2\\
&\leq\int_\Omega n\ln\frac n{\overline{n_0}}
+Mm|\Omega|\int_\Omega|w-\overline{w}|^2+Mm+m^2.
\end{split}
\end{equation}
Furthermore, if $m<\min\{4\pi,\frac1{4M|\Omega|}\}$, then the Young inequality shows that
\begin{equation}\label{3.46}
\begin{split}
\frac d{dt}
&\int_\Omega(w-\overline{w})^2+\int_\Omega(w-\overline{w})^2\\
&\leq2\int_\Omega n\ln\frac n{\overline{n_0}}+2Mm+2m^2.
\end{split}
\end{equation}
Thanks to  \eqref{3.9} and \eqref{3.15}, integrating the above inequality with respect to time over $(t_1,t_2)$ yields
 \begin{align}\label{3.47}
&\frac1{t_2-t_1}\int_{t_1}^{t_2}\int_\Omega(w-\overline{w})^2\notag\\
&\leq\frac2{t_2-t_1}\int_{t_1}^{t_2}
\int_\Omega n\ln\frac n{\overline{n_0}}+2Mm+2m^2
+\frac1{t_2-t_1}\int_\Omega(w(\cdot,t_1)-\overline{w}(\cdot,t_1))^2\notag\\
&\leq\frac{m C_1(t_1+1)}{t_2-t_1}+(8Mm^2+2m+4M)m
+(m+e^{-t_1}\int_\Omega w_0)mS_0^2(K)\\
&\quad+4m\ln_{+}\frac{|\Omega|}{m}
+\frac{\|w(\cdot,t_1)-\overline{w}(\cdot,t_1)\|_{L^2}^2}{t_2-t_1}\notag\\
&\leq\frac{(m C_1+ C_3)(t_1+1)}{t_2-t_1}+(8Mm^2+2m+4M)m
+(m+e^{-t_1}\int_\Omega w_0)mS_0^2(K)+4m\ln_{+}\frac{|\Omega|}{m}
.\notag
\end{align}
In addition, using \eqref{3.9} again, integrating \eqref{3.11} with respect to time over $(t_1,t_2)$ yields
\begin{align}\label{3.48}
&\frac1{t_2-t_1}\int_{t_1}^{t_2}\int_\Omega|u|^2\notag\\
&\leq\frac{C_p}{t_2-t_1}\int_{t_1}^{t_2}\int_\Omega|\nabla u|^2\notag\\
&\leq\frac{C_p}{t_2-t_1}\int_\Omega|u(\cdot,t_1)|^2
+\frac{C_2 C_pm}{t_2-t_1}\int_{t_1}^{t_2}\int_\Omega
 n\ln\frac n{\overline{n_0}}+C_2 C_pm^2\\
&\leq\frac{C_p\|u(\cdot,t_1)\|_{L^2}^2}{t_2-t_1}
+C_2 C_pm^2\notag\\
&\quad+C_2 C_pm\bigg\{\frac{C_1(t_1+1)}{t_2-t_1}+(4Mm^2+M)m
%+(\frac{m^2}{\pi}+\frac{m}{\pi}e^{-t_1}\int_\Omega w_0)S_0(\|v_0\|_{L^\infty(\Omega)})+2m\ln_{+}\frac{|\Omega|}{m}\bigg\}.\notag\\
+(m+e^{-t_1}\int_\Omega w_0)mS_0^2(K)+2m\ln_{+}\frac{|\Omega|}{m}\bigg\}.\notag
\end{align}
On the other hand,  similar to the proof  of \eqref{3.15},  thanks to \eqref{3.9}, one can conclude that there exists
 $C_4=C_4(m)>0$ such that  for all $t>0$
\begin{equation*}
\int_\Omega |u|^2(\cdot,t)
\leq C_4 (t+1),
\end{equation*}
and hence
 \begin{align}\label{3.49}
\frac1{t_2-t_1}\int_{t_1}^{t_2}\int_\Omega|u|^2&\leq\frac{C_p(C_4+m C_1C_2)(t_1+1)}{t_2-t_1}+
C_2 C_pm\bigg\{(4Mm^2+M+1)m\notag\\
&\quad+(m+e^{-t_1}\int_\Omega w_0)mS_0^2(K)+2m\ln_{+}\frac{|\Omega|}{m}\bigg\}.
 \end{align}
Now from \eqref{3.9}, \eqref{3.4}, (\ref{3.47}) and \eqref{3.49}, it follows that there is $C_5=C_5(m)>0$ such that
\begin{equation}\label{3.50}
\begin{split}
&\frac1{t_2-t_1}\bigg\{\int_{t_1}^{t_2}
\int_\Omega n\ln\frac n{\overline{n_0}}
+\frac1 2\int_{t_1}^{t_2}\int_\Omega|\nabla z|^2
+\int_{t_1}^{t_2}\int_\Omega(w-\overline{w})^2
+\int_{t_1}^{t_2}\int_\Omega |u|^2\bigg\}\\
\leq&\frac{C_5(t_1+1)}{t_2-t_1}+
(3+C_2C_pm)
((4Mm^2+M)m
+2m\ln_{+}\frac{|\Omega|}{m})\\
&
+
(2M+1)m+
\{C_2C_p+2+(2+C_2C_pm)S_0^2(K)\} m^2\\
&
+\{(2+C_2C_pm)mS_0^2(K)+1\}\int_\Omega w_0 \cdot e^{-t_1}
\end{split}
\end{equation}
for any $0<t_1<t_2 $.

Due to $\lim\limits_{m\to 0}m\ln_+\frac{|\Omega|}{m}=0$, one can pick $m_0<m_*$ sufficiently small such that
$$
(3+C_2C_pm)
((4Mm^2+M)m
+2m\ln_{+}\frac{|\Omega|}{m})
+
(2M+1)m+
\{C_2C_p+2+(2+C_2C_pm)S_0^2(K)\} m^2
<\frac \delta 2,
$$
for all $m<m_0$.
At the same time, we can choose $t_1>0$ large enough such that
$$\{(2+C_2C_pm)mS_0^2(K)+1\}\int_\Omega w_0 \cdot e^{-t_1}
\leq\frac\delta 4,$$ and  then fix $t_2>t_1$ fulfilling $\frac{C_5(t_1+1)}{t_2-t_1}<\frac \delta 4$.
Hence, \eqref{3.15} implies that the functional $\mathcal{F}$ defined through \eqref{3.18} satisfies
\begin{equation}\label{3.51}
\frac1{t_2-t_1}\int_{t_1}^{t_2}\mathcal{F}(t)dt\leq\delta,
\end{equation}
which allows us to find some $\hat{t}_0=\hat{t}_0(n_0,v_0,w_0,u_0,\varepsilon)\in(t_1,t_2)$ such that  $\mathcal{F}(\hat{t}_0)<\delta$ is valid.
Subsequently, as an application of Lemma \ref{lem3.5} and due to $t_2\geq\hat{t}_0$,
 we get  $\mathcal{F}(t_2)<\delta$ immediately. Applying Lemma \ref{lem3.5} to $t_0:=t_2$ once more, we
     %by $\delta<1$, we
     readily obtain \eqref{3.35}--\eqref{3.38} and \eqref{3.39}--\eqref{3.42} with some appropriately large $C>0$ and $t_0^*=t_2$.
\end{proof}

Next, we will give a series of estimates pertaining to the exponential stabilization properties of the solutions to \eqref{2.4} via a bootstrap procedure.
We remark that the derived estimates are shown to be independent of the parameter $\varepsilon>0$, which allows us to extend the convergence arguments for
\eqref{2.4} to \eqref{1.6}.
%The convergence properties established in Lemma \ref{lemma4.4} will play a critical role.
As a starting point of the bootstrap procedure,  we utilize Lemma \ref{lem3.6} to establish time-independent bounds for $\int_\Omega n^2$, $\int_\Omega|\nabla u|^2$ and $\int_{\Omega}|\nabla z|^{4}$, which %in view of Lemma \ref{lemma4.6}
will pave the way for subsequent regularity estimates that are needed to complete the proof of Theorem 1.1.

\begin{lemma}\label{lem3.7}
There exists $C>0$ such that if $m_*$ is as in Lemma \ref{lem3.6} and \eqref{1.8} holds with $\int_\Omega n_0<m_*$, then we have
\begin{equation}\label{3.52}
\int_\Omega n^2(\cdot,t)\leq C\mbox{ for all }t>t_1^*:=t_0^*+1.
\end{equation}
\end{lemma}
\begin{proof}[Proof]
From Lemma \ref{lem3.6}, it follows that there is $C_1>0$  such that
\begin{equation}\label{3.53}
\int_\Omega|\nabla z|^2\leq 1\,\,\mbox{ for all }\,t>t_0^*,
\end{equation}
\begin{equation}\label{3.54}
\int_{t_0^*}^\infty\int_\Omega\frac{|\nabla n|^2}{n}\leq C_1
\end{equation}
and
\begin{equation}\label{3.55}
\int_{t_0^*}^\infty\int_\Omega|\Delta z|^2\leq C_1.
\end{equation}
Moreover, we employ the Sobolev inequality and the Gagliardo-Nirenberg inequality along with elliptic theory to find $C_2>0,C_3>0$ and $C_4>0$ fulfilling
\begin{equation}\label{3.56}
\int_\Omega\varphi^2\leq C_2\bigg\{\int_\Omega|\nabla\varphi|\bigg\}^2
+C_2\bigg\{\int_\Omega|\varphi|\bigg\}^2\mbox{ for all }\varphi\in W^{1,1}(\Omega)
\end{equation}
and
\begin{equation}\label{3.57}
\|\varphi\|_{L^4}^2\leq C_3\|\nabla\varphi\|_{L^2}\|\varphi\|_{L^2}
+C_3\|\varphi\|_{L^2}^2\mbox{ for all }\varphi\in W^{1,2}(\Omega)
\end{equation}
as well as
\begin{equation}\label{3.58}
\|\nabla\varphi\|_{L^4}^2\leq C_4\|\Delta\varphi\|_{L^2}
\|\nabla\varphi\|_{L^2}
\mbox{ for all }\varphi\in W^{2,2}(\Omega)
\mbox{ such that }
\frac{\partial\varphi}{\partial\nu}=0
\mbox{ on }\partial\Omega.
\end{equation}

For any given  $t>t_0^*+1$, one can find $t_0=t_0(n_0,v_0,w_0,u_0,\varepsilon)\in(t-1,t)$ in such a way that in accordance with (\ref{3.54}),
\begin{equation}
\label{3.59}
\int_\Omega\frac{|\nabla n(\cdot,t_0)|^2}{n(\cdot,t_0)}\leq C_1.
\end{equation}

On the other hand, testing the first equation in (\ref{3.2}) by $n$ and once more using that $\nabla\cdot u=0$ and \eqref{1.7b}, we obtain
\begin{equation}\label{3.60}
\begin{split}
\frac1 2\frac d{dt}\int_\Omega n^2+\int_\Omega|\nabla n|^2
&=\int_\Omega\nabla\cdot(nvS(\cdot,n,v)\cdot\nabla z)n\\
&\leq S_0(\|v_{0}\|_{L^{\infty}(\Omega)})\int_\Omega|n\nabla z\cdot\nabla n|\\
&\leq\frac1 2\int_\Omega|\nabla n|^2
+\frac1 2S^2_0(K)\int_\Omega n^2|\nabla z|^2\\
\end{split}
\end{equation}
for all $t>0$.

By \eqref{3.57}, \eqref{3.58} and the Young inequality, we can estimate
\begin{equation*}
\begin{split}
&\frac1 2S^2_0(K)\int_\Omega n^2|\nabla z|^2\\
&\leq\frac1 2S^2_0(K)\|n\|_{L^4}^2\|\nabla z\|_{L^4}^2\\
&\leq\frac{C_3C_4S^2_0(K)}{2}
\bigg\{\|\nabla n\|_{L^2}\|n\|_{L^2}
+\|n\|_{L^2}^2\bigg\}\cdot\|\Delta z\|_{L^2}\|\nabla z\|_{L^2}\\
&\leq\frac{C_3C_4}{2}S^2_0(K)\|\nabla n\|_{L^2}\|n\|_{L^2}\|\Delta z\|_{L^2}
+\frac{C_3C_4}{2}S^2_0(K)\|n\|_{L^2}^2\|\Delta z\|_{L^2}\\
&\leq\frac 1 2\int_\Omega|\nabla n|^2
+\frac{C_3^2C_4^2}{8}S^4_0(K)
\bigg\{\int_\Omega|\Delta z|^2\bigg\}\int_\Omega n^2
+\frac{C_3C_4}{2}S^4_0(K)
\bigg\{\int_\Omega|\Delta z|^2\bigg\}^{\frac1 2}\int_\Omega n^2\\
&\leq\frac 1 2\int_\Omega|\nabla n|^2
+\frac{C_3^2C_4^2}{4}S^4_0(K)
\bigg\{\int_\Omega|\Delta z|^2\bigg\}\int_\Omega n^2
+\frac1{2}S^4_0(K)\int_\Omega n^2
\end{split}
\end{equation*}
for all $t>0$. Therefore, from (\ref{3.60}) we obtain that
\begin{equation}\label{3.61}
\frac d{dt}\int_\Omega n^2\leq S^4_0(K)\bigg\{\frac{C_3^2C_4^2}{2}\int_\Omega|\Delta z|^2+1\bigg\}
\int_\Omega n^2\mbox{ for all } t>t_0^*.
\end{equation}
Now thanks to \eqref{3.56} and \eqref{3.59},  an application of the Cauchy-Schwarz inequality leads to
\begin{equation}
\begin{split}
\int_\Omega n(\cdot,t_0)^2
&\leq C_2\bigg\{\int_\Omega|\nabla n(\cdot,t_0)|\bigg\}^2
+C_2\bigg\{\int_\Omega|n(\cdot,t_0)|\bigg\}^2\\
&\leq C_2\bigg\{\int_\Omega n(\cdot,t_0)\bigg\}
\cdot\int_\Omega\frac{|\nabla n(\cdot,t_0)|^2}{n(\cdot,t_0)}
+C_2\bigg\{\int_\Omega|n(\cdot,t_0)|\bigg\}^2\\
&\leq C_5:=C_1C_2m_*+C_2m_*^2,
\end{split}
\end{equation}
hence integrating \eqref{3.61} with respect to time over $(t_0,t)$, we arrive at
\begin{equation*}
\begin{split}
\int_\Omega n(\cdot,t)^2
&\leq\bigg\{\int_\Omega n(\cdot,t_0)^2\bigg\}
\cdot\exp
\bigg\{S^4_0(K)\frac{C_3^2C_4^2}{2}\int_{t_0}^{t}
\int_\Omega|\Delta z|^2+(t-t_0)\bigg\}\\
&\leq C_5\cdot\exp\bigg\{S^4_0(K)\frac{C_1C_3^2C_4^2}{2}+1\bigg\},
\end{split}
\end{equation*}
due to $t-t_0\leq 1$. As $t>t_0^*+1$ is arbitrary, this establishes (\ref{3.52}) with $t_1:=t_0^*+1$ and $C:=C_5\exp\{S^4_0(K)C_1C_3^2C_4^2+1\}$.
\end{proof}

Furthermore, on the basis of the bounds provided by Lemma 3.7, one can improve the regularity of the fluid variable $u$
by  means of a suitably organized bootstrap procedure.

\begin{lemma}\label{lem3.8}
Let $\beta\in(\frac1 2,1)$. Then one can find $C>0$ and $C(\beta)>0$  such that for all $(n_0,v_0,w_0,u_0)$ satisfying (\ref{1.8}) with $\int_\Omega n_0<m_0$,
\begin{equation}\label{3.63}
\int_\Omega|\nabla u(\cdot,t)|^2\leq C
\end{equation}
as well as
$$\int_\Omega|A^\beta u(\cdot,t)|^2 \leq C(\beta)
$$
for all $t>t_1^*$ which is  given as  in  Lemma \ref{lem3.7}.
\end{lemma}
\begin{proof}[Proof] Thanks to \eqref{3.42} and \eqref{3.52}, the estimates claimed in this lemma can be achieved in
line with the reasonings of \cite[Lemmas 3.4--3.5]{Winkler(2020)}.
\end{proof}

The regularity information gained in  \eqref{3.52}, \eqref{3.41} and \eqref{3.63} allows for the establishing  %
% of the fluid field $u$ achieved in Lemma \ref{lem3.8} guaranteed is enough to allow for
of the bounds of  $\nabla w$  in Lebesgue spaces with arbitrarily high integrability powers
%$with any $q<\infty$
through a standard  bootstrap argument.

\begin{lemma}\label{lem3.8b} There is  $C>0$   such that for all $(n_0,v_0,w_0,u_0)$ fulfilling \eqref{1.8} with $\int_\Omega n_0<m_0$,
\begin{equation}\label{3.63b}
\int_\Omega|\nabla w(\cdot,t)|^q\leq C
\end{equation}
for all $t>t_1^*$ which is  as  in  Lemma \ref{lem3.7}.
\end{lemma}

\begin{proof}[Proof] Re-write the $w$-equation in \eqref{3.2} in the form
 $$
 w_t=\triangle w -w+g $$
  with $g(x,t):=n-u\cdot \nabla w$. In the first step, since \eqref{3.41} in particular provides some $C_1>0$ such that
    $$
  \int^{t+1}_t\int_\Omega|\nabla w(\cdot,t)|^2\leq C_1
  $$
  for all $t\geq t_1^*$, we get  the bound for $\int_\Omega|\nabla w(\cdot,t)|^2$ via
a $L^2$-testing procedure. Thanks to  \eqref{3.52} and \eqref{3.63}, we  obtain the existence of  $C_2$ such that
 $\int_\Omega|g(\cdot,t)|^2\leq C_2$ for all $t\geq t_1^*$, and thereupon \eqref{3.63b} results from the application of
 regularization estimates for the Neumann heat semigroup  $(e^{t\Delta})_{t\geq0}$.
\end{proof}

Inter alia as an other important consequence of  \eqref{3.63}, it allows us to derive a bound for $\int_\Omega|\nabla z(\cdot,t)|^4$.

\begin{lemma}\label{lem3.9}
There exist $C>0$ and  $t_2^*=t_2^*(n_0,v_0,w_0,u_0)>t_1^*$  such that for all $(n_0,v_0,w_0,u_0)$ satisfying (\ref{1.8}) with $\int_\Omega n_0<m_0$,
\begin{equation}\label{3.64}
\int_\Omega|\nabla z(\cdot,t)|^4\leq C \,\mbox{ for all }t\geq t_2^*.
\end{equation}
%where $T$ is given as Lemma \ref{lem3.7}.
\end{lemma}
\begin{proof}[Proof]
Testing the second equation in (\ref{3.2}) against $\nabla z\cdot|\nabla z|^2$ and  using the identity $\nabla z\cdot\nabla\Delta z=\frac1 2\Delta|\nabla z|^2-|D^2z|^2$, we obtain
\begin{align}\label{3.65}
&\frac d{dt}\int_\Omega|\nabla z|^4
+2\int_\Omega|\nabla|\nabla z|^2|^2\notag\\
\leq & 12\int_\Omega|\nabla z|^6
+36\int_\Omega w^2|\nabla z|^2
+4\int_\Omega|\nabla z|^4|\nabla u|
+C_1(\int_\Omega|\nabla z|^2)^2.
\end{align}
 By  the Gagliardo-Nirenberg inequality and the Young inequality, we can see that
\begin{equation}\label{3.66}
\begin{split}
\frac d{dt}&\int_\Omega|\nabla z|^4
+2\int_\Omega|\nabla|\nabla z|^2|^2
+\int_\Omega|\nabla z|^4\\
&\leq 12\int_\Omega|\nabla z|^6
+36\int_\Omega w^2|\nabla z|^2
+4\int_\Omega|\nabla z|^4|\nabla u|
+C_1(\int_\Omega|\nabla z|^2)^2
+\int_\Omega|\nabla z|^4\\
&\leq 12\int_\Omega|\nabla z|^6
+36\int_\Omega w^2|\nabla z|^2
+4\int_\Omega|\nabla z|^4|\nabla u|+\frac 12\int_\Omega|\nabla|\nabla z|^2|^2+C_2
\end{split}
\end{equation}
with some $C_2>0$ for all $t>t_1^*$ as provided by  Lemma \ref{lem3.7}.

Here in the first integral on the right side of \eqref{3.65}, we invoke  the Gagliardo-Nirenberg inequality to find $C_{GN}>0$ such that
\begin{equation}
\int_\Omega|\nabla z|^6
\leq C_{GN}\big(\int_\Omega|\nabla|\nabla z|^2|^2\big)
(\int_\Omega|\nabla z|^2)
+C_{GN}(\int_\Omega|\nabla z|^2)^3.
\end{equation}
In addition, thanks to \eqref{3.63} and \eqref{3.53}, the application of Gagliardo-Nirenberg's inequality and the Young inequality
gives us
\begin{equation}
\begin{split}
4\int_\Omega|\nabla z|^4|\nabla u|
&\leq 4 \||\nabla z|^2\|_{L^4}^2 \|\nabla u\|_{L^2}\\
&\leq C_3C_{GN}\big(\|\nabla|\nabla z|^2\|_{L^2}^\frac 32
\||\nabla z|^2\|_{L^1}^{\frac1 2}
+\||\nabla z|^2\|_{L^1}^2\big)\\
%&\leq C-3 C_{GN}\big(\varepsilon\|\nabla|\nabla z|^2\|_{L^2}^2+(C(\varepsilon)+1)\||\nabla z|^2\|_{L^1}^{2}\big)\|\nabla u\|_{L^\frac3 2}\\
&\leq \frac12 \|\nabla|\nabla z|^2\|_{L^2}^2+C_3.\\
%&\leq \varepsilon C_{GN}\|\nabla|\nabla z|^2\|_{L^2}^2+C_3\\
\end{split}
\end{equation}
By the Young inequality, we estimate
\begin{equation}\label{3.70}
\begin{split}
36\int_\Omega w^2|\nabla z|^2
&\leq\int_\Omega|\nabla z|^6+216\int_\Omega w^3\\
&\leq\int_\Omega|\nabla z|^6+216C_{GN}
(\|\nabla w\|_{L^2}^2\|w\|_{L^1}+\|w\|_{L^1}^3)\\
&\leq\int_\Omega|\nabla z|^6+
C_4\|\nabla w\|_{L^2}^2+C_4
\end{split}
\end{equation}
with $C_4:=\max\{216C_{GN}(\int_\Omega w_0+m),216(\int_\Omega w_0+m)^3\}$.

The combination of the above inequality with (\ref{3.66}) then yields
\begin{equation}
\begin{split}\label{3.71}
\frac d{dt}&\int_\Omega|\nabla z|^4
+(1-14C_{GN}(\int_\Omega|\nabla z|^2))\int_\Omega|\nabla|\nabla z|^2|^2
+\int_\Omega|\nabla z|^4\\
&\leq C_4\int_\Omega |\nabla w|^2
+C_5\\
\end{split}
\end{equation}
for some $C_5>0$ and $t>t_1^*$.

Thanks to the decay property of $\int_\Omega|\nabla z|^2(\cdot,t)$ in (\ref{3.36}),
we fix  $t_2^*:=t_1^*+(\frac{\ln(28C_{GN})}{\kappa})_+$ such that
\begin{equation}
14C_{GN}\int_\Omega|\nabla z|^2(\cdot,t)\leq \frac12 \mbox{ for all }t\geq t_2^*,
\end{equation}
and hence
\begin{equation}\label{3.72}
\begin{split}
\frac d{dt}\int_\Omega|\nabla z|^4
+\frac12\int_\Omega|\nabla|\nabla z|^2|^2
+\int_\Omega|\nabla z|^4
\leq
C_4\int_\Omega |\nabla w|^2+C_5.
 \end{split}
\end{equation}
Furthermore, in light of   \eqref{3.40} and \eqref{3.36}, an application of \eqref{3.58} leads to
 $$
 \int^{t+1}_t\int_\Omega |\nabla z|^4\leq C_6
  $$
  with some $C_6>0$ for any $t\geq t_0^*$, which together with \eqref{3.63b} yields
  %thereby an application of Lemma 3.4 in  \cite{WinklerJFA(2019)} implies that
  %there exists $C_7>0$ fulfilling
\begin{equation}\label{3.73}
\int_\Omega|\nabla z(\cdot,t)|^4\leq C_7
\end{equation}
 for all  $t\geq t_2^*$ and thus completes the proof.
\end{proof}

Thanks to the estimates  of $\int_{\Omega}|\nabla z|^{4}$  and $\int_\Omega|A^\beta u|^2$ ($\beta\in (\frac 12,1)$) established in Lemmas \ref{lem3.8} and \ref{lem3.9}, respectively, we can now derive the boundedness  of $(n,u)$ in  $L^{\infty}(\Omega)\times L^{\infty}(\Omega)$ and $(\nabla z, \nabla w)$ in $L^{q}(\Omega)\times L^{\infty}(\Omega)$ respectively by  the standard smoothing properties of the Neumann heat semigroup.

\begin{lemma}\label{lem3.10}
Let $m_0$ and $t_2^*$ be as in Lemma \ref{lem3.6} and Lemma \ref{lem3.9} respectively. There is $C>0$ such that whenever (\ref{1.8}) is valid with $\int_\Omega n_0<m_0$, we have
\begin{equation}\label{3.74}
\|n(\cdot,t)\|_{L^\infty(\Omega)}\leq C,
\end{equation}
\begin{equation}\label{3.75}
\|v(\cdot,t)\|_{W^{1,\infty}(\Omega)}\leq C,
\end{equation}
\begin{equation}\label{3.76}
\|w(\cdot,t)\|_{W^{1,\infty}(\Omega)}\leq C
\end{equation}
%as well as
\begin{equation}\label{3.77}
\|u(\cdot,t)\|_{L^\infty(\Omega)}\leq C
\end{equation}
 for all $t>t_2$.
\end{lemma}
\begin{proof}[Proof] Due to the continuity of the embedding $D(A^\beta)\hookrightarrow L^\infty(\Omega)$ with $\beta>\frac 12$,
with the help of Lemma \ref{lem3.8} we can find $C>0$ fulfilling $\|u(\cdot,t)\|_{L^\infty(\Omega)}\leq C$ for all $t>t_2^*$.
On the basis of \eqref{3.52} and \eqref{3.64},
 one can adapt an approach illustrated in e.g.~\cite[Lemma 4.4] {Winkler(2016)} or \cite[Lemma 5.1]{Lankeit(2020)} to obtain
\eqref{3.74}. Thereafter invoking  the standard smoothing properties of the Neumann heat semigroup
(see e.g. Lemma \ref{lemma2.1}), we can obtain
  \eqref{3.75} and \eqref{3.76} readily.
\end{proof}

Now we are ready to  prove the first part of Theorem 1.1.
To establish the global boundedness of classical solutions to \eqref{1.6},
  our approach is to make use of the a priori bounds of
 the solution $(n_\varepsilon, v_\varepsilon, w_\varepsilon,  u_\varepsilon, P_\varepsilon) $ to the regularized problem
  \eqref{2.4} on $(0,t_2^*)$, where $t_2^*>0$  is given in Lemma \ref{lem3.9}. Here it should be underlined that these bounds are independent of the regularization index $\varepsilon$.
 % In the context of this,
% this direct,  it  to  derive  $\varepsilon$-independent $L^2((0,t_2),L^2(\Omega))$  bound for  $n$.
\begin{lemma}\label{lemma3.11} Let the  assumptions of Theorem 1.1 hold. Then the problem \eqref{1.6} possesses
a global classical solution which is bounded and enjoys decay properties stated in Lemma \ref{lem3.6}.
\end{lemma}
\begin{proof}[Proof] From Lemma \ref{lemma3.4}, there exists $C_1=C_1(t_2^*)$ such that
$
\|w_\varepsilon(\cdot,t)\|_{L^2(\Omega)}\leq C_1
$ for all $t\leq t_2^*$. Thereafter proceeding along with a well-established Moser-Alikakos iterative technique, one can
derive the $t_2^*$-dependent $L^\infty$-bound of $z$ (\cite[Lemma 4.5]{Dai Liu(2023)}), which rules out the singularity of $S$ at $v=0$ on
the interval  $(0,t_2^*)$.
Furthermore, in line with the reasonings of \cite[Lemma 3.3--3.5]{Winkler(2020)} or \cite[Proposition 4.1]{Dai Liu(2023)} we are able to establish
 \begin{equation}\label{3.78}
\|n_\varepsilon(\cdot,t)\|_{L^\infty(\Omega)}+
\|v_\varepsilon(\cdot,t)\|_{W^{1,\infty}(\Omega)}+
\|w_\varepsilon(\cdot,t)\|_{W^{1,\infty}(\Omega)}+\|A^{\beta}u_\varepsilon(\cdot,t)\|_{L^\infty(\Omega)}\leq C_2:=C(t_2^*)
\end{equation}
for all $\beta\in (\frac 12,1)$ and $t\leq t_2^*$, which together with Lemma \ref{lem3.10} implies that \eqref{3.74}--\eqref{3.77}
is valid for all $t>0$.  Now the corresponding  H\"{o}lder regularity of $(n_\varepsilon, v_\varepsilon, w_\varepsilon,  u_\varepsilon)$ can be realized
based on a combination of the parabolic Schauder theory and bootstrap argument (see \cite[Lemma 5.7]{CaoLankeit}). Furthermore, invoking the  Arzel\`{a}-Ascoli theorem,
there exists a subsequence $(\varepsilon_j)_{j\in N} \in(0, 1)$ such that $\varepsilon_j \rightarrow 0$ as $j\rightarrow \infty$ and
upon $\varepsilon_j\searrow 0$, $(n_\varepsilon, v_\varepsilon, w_\varepsilon,u_\varepsilon)$ converges to
 $(n, v,w,u)$ in $C_{loc}^{2+\gamma, 1+\frac \gamma 2}(\overline \Omega \times (0,\infty))$, resulting in a classical solution to the problem \eqref{1.6}
 (\cite[Lemma 5.7]{CaoLankeit}).
Finally one can achieve the decay properties stated in Lemma \ref{lem3.6} for $(n,v,w,u)$ by the lower semi-continuity of the respective
norms (see \cite[Lemma 5.6]{CaoLankeit}).\end{proof}

\section{Asymptotic Behavior}
In this section, we prove large time convergence for each component of the solution to \eqref{1.6} rather
the approximated system \eqref{2.4}.  To this end, we shall turn the information in \eqref{3.37} into
 the decay property of $\|\nabla z(\cdot,t)\|_{L^q(\Omega)}$ for any finite $q$
 by employing the smoothing properties of the Neumann heat semigroup.
\begin{lemma}\label{lemma4.1}
Let $m_0$ and $t_2^*$ be as in Lemma \ref{lem3.6} and Lemma \ref{lem3.9}, respectively. Then for any $4<q<\infty$, there exist $\kappa=\kappa(q)>0$ and $C=C(q)>0$ such that whenever (\ref{1.8}) is valid with $\int_\Omega n_0<m_0$, we have
\begin{equation}\label{4.1}
\|\nabla z(\cdot,t)\|_{L^q(\Omega)}\leq Ce^{-\kappa t}\mbox{ for all }t>t_2^*+1.
\end{equation}
\end{lemma}
\begin{proof}[Proof]
Applying the variation-of-constants formula to the second equation of (\ref{3.2}), we have
\begin{equation*}
\begin{split}
z(\cdot,t)
&=e^{(t-t_2)\Delta}z(\cdot,t_2)
-\int_{t_2}^te^{(t-s)\Delta}|\nabla z|^2(\cdot,s)ds\\
&\quad+\int_{t_1}^te^{(t-s)\Delta}w(\cdot,s)ds
-\int_{t_2}^t e^{(t-s)\Delta}(u \cdot\nabla z)(\cdot,s)ds
\end{split}
\end{equation*}
for $t\geq t^*_2$,
%Thanks to  Lemma \ref{lem3.6} and Lemma \ref{lem3.9},
which along with the application of Lemma \ref{lemma2.1} leads to
\begin{equation}\label{4.2}
\begin{split}
&\|\nabla z(\cdot,t)\|_{L^q(\Omega)}\\
&\leq C_1e^{-\lambda_1(t-t_2^*)}
\|\nabla z(\cdot,t_2^*)\|_{L^4(\Omega)}
+\int_{t_2^*}^t
\|\nabla e^{(t-s)\Delta}|\nabla z|^2(\cdot,s)\|_{L^q(\Omega)}ds\\
&\quad+\int_{t_2^*}^t\|\nabla e^{(t-s)\Delta}
\big(w(\cdot,s)-\overline{w}(s)\big)\|_{L^q(\Omega)}ds
+\int_{t_2^*}^t\|\nabla e^{(t-s)\Delta}
(u\cdot\nabla z)(\cdot,s)\|_{L^q(\Omega)}ds\\
&\leq C_1e^{-\lambda_1(t-t_2^*)}
\|\nabla z(\cdot,t_2^*)\|_{L^4(\Omega)}
+C_1\int_{t_2^*}^t(1+(t-s)^{-1+\frac1 q})
e^{-\lambda_1(t-s)}\|\nabla z(\cdot,s)\|_{L^4(\Omega)}^2ds\\
&\quad+C_1\int_{t_2^*}^t(1+(t-s)^{-1+\frac1 q})
e^{-\lambda_1(t-s)}
\|w(\cdot,s)-\overline{w}(\cdot,s)\|_{L^2(\Omega)}ds\\
&\quad+C_1\int_{t_2^*}^t(1+(t-s)^{-\frac3 4+\frac1 q})
e^{-\lambda_1(t-s)}\|u(\cdot,t)\|_{L^\infty(\Omega)}
\|\nabla z(\cdot,s)\|_{L^4(\Omega)}ds\\
&\leq C_1\|\nabla z(\cdot,t_2^*)\|_{L^4(\Omega)}e^{-\lambda_1(t-t_2^*)}
+C_1\int_0^\infty(1+\sigma^{-1+\frac1 q})
e^{-\lambda_1\sigma}d\sigma
\cdot\sup\limits_{t\geq t_2^*}
\|\nabla z(\cdot,t)\|_{L^4(\Omega)}^2\\
&\quad+C_1\int_0^\infty(1+\sigma^{-1+\frac1 q})
e^{-\lambda_1\sigma}d\sigma
\cdot\sup\limits_{t\geq t_2^*}
\|w(\cdot,t)-\overline{w}(\cdot,t)\|_{L^2(\Omega)}\\
&\quad+C_1\int_0^\infty(1+\sigma^{-\frac3 4+\frac 1 q})
e^{-\lambda_1\sigma}d\sigma
\cdot\sup\limits_{t\geq t_2^*}\|\nabla z(\cdot,t)\|_{L^4(\Omega)}.
\end{split}
\end{equation}
Together with \eqref{3.64} and \eqref{3.37}, this implies that there exists $C_2(q)>0$ fulfilling
\begin{equation}\label{4.3}
\|\nabla z(\cdot,t)\|_{L^q(\Omega)}\leq C_2(q) \,\mbox{ for all } t>t_2^*+1.
\end{equation}
Now by the interpolation inequality, we have
  \begin{equation*}
\|\nabla z(\cdot,t)\|_{L^q(\Omega)}\leq
\|\nabla z(\cdot,t)\|^{\frac{q-2}{q-1}}_{L^{2q}(\Omega)}
\|\nabla z(\cdot,t)\|^{\frac{1}{q-1}}_{L^2(\Omega)},
\end{equation*}
which together with  \eqref{4.3} and \eqref{3.36} yields  $C_3(q)>0$ such that
 \begin{equation*}
\|\nabla z(\cdot,t)\|_{L^q(\Omega)}\leq C_3(q) e^{-\frac \kappa{2(q-1)}t}
\end{equation*}
  for any $t>t_2^*+1$.
\end{proof}

\begin{lemma}\label{lemma4.3}
Let $m_0>0$ be as in Lemma \ref{lem3.6} and suppose that (\ref{1.8}) is valid with $\int_{\Omega}n_0<m_0$. Then we have
\begin{equation}\label{4.4}
n(\cdot,t)\to\overline{n_0}\mbox{ in }L^\infty(\Omega),\quad w(\cdot,t)\to\overline{n_0}\mbox{ in }L^\infty(\Omega)
\end{equation}
exponentially as $t\to\infty$.
\end{lemma}
\begin{proof}[Proof]
Applying the variation-of-constants formula for $n$ in \eqref{3.2}, we obtain from Lemma \ref{lemma2.1} that for all $t>t_2^*+2$
\begin{align}\label{4.5}
&\|n(\cdot,t)-\overline{n_0}\|_{L^\infty(\Omega)}\notag\\
&\leq\|e^{(t-t_2^*-1)\Delta}
(n(\cdot,t_2^*+1)-\overline{n_0})\|_{L^\infty(\Omega)}
+\int_{t_2^*+1}^t
\|e^{(t-s)\Delta}\nabla\cdot(nS(\cdot,n,v)v\cdot \nabla z)(\cdot,s)\|_{L^\infty(\Omega)}ds\notag\\
&\quad+\int_{t_2^*+1}^t
\|e^{(t-s)\Delta}\nabla\cdot(un)(\cdot,s)\|_{L^\infty(\Omega)}ds.
\end{align}
From  Lemma \ref{lemma2.1}, \eqref{3.74},  \eqref{4.1}  and  Lemma \ref{lemma2.5},  we infer
\begin{equation}\label{4.6}
\begin{split}
&\int_{t_2^*+1}^t
\|e^{(t-s)\Delta}\nabla\cdot(nS(\cdot,n,v)v\cdot \nabla z)(\cdot,s)\|_{L^\infty(\Omega)}ds\\
\leq &
C_1  S_0(\|v_{0}\|_{L^{\infty}(\Omega) })\int_{t_2^*+1}^t(1+(t-s)^{-\frac3 4})e^{-\lambda_1(t-s)}
\|n\nabla z(\cdot,s)\|_{L^4(\Omega)}ds\\
 \leq& C_1S_0(\|v_{0}\|_{L^{\infty}(\Omega) })\int_{t_2^*+1}^t(1+(t-s)^{-\frac3 4})
e^{-\lambda_1(t-s)}\|n(\cdot,s)\|_{L^\infty(\Omega)}
\|\nabla z(\cdot,s)\|_{L^4(\Omega)}ds\\
\leq &C_2\int_0^t(1+(t-s)^{-\frac3 4})
e^{-\lambda_1(t-s)}e^{-\frac\kappa 6 s}ds\\
\leq &C_3 e^{
-\min
\{
\frac\kappa 6, \lambda_1
\} t
}.
\end{split}
\end{equation}
In addition,
from  Lemma \ref{lemma2.1}, \eqref{3.74},  \eqref{3.38}, \eqref{3.63}  and  Lemma \ref{lemma2.5},  it follows that
\begin{equation}\label{4.7}
\begin{split}
&
\int_{t_2^*+1}^t
\|e^{(t-s)\Delta}\nabla\cdot(un)(\cdot,s)\|_{L^\infty(\Omega)}ds\\
\leq &
C_4\int_{t_2^*+1}^t(1+(t-s)^{-\frac3 4})
e^{-\lambda_1(t-s)}\|n(\cdot,s)\|_{L^\infty(\Omega)}
\|u(\cdot,s)\|_{L^4(\Omega)}ds\\
\leq &
C_5\int_{t_2^*+1}^t(1+(t-s)^{-\frac3 4})
e^{-\lambda_1(t-s)}\|u(\cdot,s)\|_{L^4(\Omega)}ds\\
\leq &
C_6\int_{t_2^*+1}^t(1+(t-s)^{-\frac3 4})
e^{-\lambda_1(t-s)}\|\nabla u(\cdot,s)\|_{L^2(\Omega)}^{\frac12}
\|u(\cdot,s)\|_{L^2(\Omega)}^{\frac12}ds\\
\leq &
C_7\int_0^t(1+(t-s)^{-\frac3 4})
e^{-\lambda_1(t-s)}e^{-\frac\kappa 2  s}ds\\
\leq &
C_8 e^{-\min
\{
\frac\kappa 2, \lambda_1
\} t
}.
\end{split}
 \end{equation}
 Inserting  \eqref{4.7} and \eqref{4.6} into \eqref{4.5},
 we obtain that for all $t>t_2^*+2$
 \begin{equation}
 \label{4.8}
\|n(\cdot,t)-\overline{n_0}\|_{L^\infty(\Omega)}\leq C_9 e^
{-\min
\{
\frac\kappa 6, \lambda_1
\} t
}.
\end{equation}

On the other hand, by the Gagliardo-Nirenberg inequality, we have
\begin{equation}\label{4.9}
\begin{split}
\|w(\cdot,t)-\overline{n_0}\|_{L^\infty(\Omega)}
&\leq\|w(\cdot,t)-\overline{w}(t)\|_{L^\infty(\Omega)}
+|\overline{w}(t)-\overline{n_0}|\\
&\leq\|\nabla w(\cdot,t)\|_{L^2(\Omega)}^{\frac3 4}
\|w(\cdot,t)-\overline{w}(t)\|_{L^2(\Omega)}^{\frac1 4}
+|\overline{w}(t)-\frac m{|\Omega|}|.
\end{split}
\end{equation}
Integrating the third equation of (\ref{3.2}) over $\Omega$ yields
\begin{equation*}
\frac d{dt}(\overline{w}-\frac m{|\Omega|})=-(\overline{w}- \frac m{|\Omega|})
\end{equation*}
and thus
\begin{equation}\label{4.10}
(\overline{w}-\frac m{|\Omega|})(t)=\bigg(\overline{w}(t_2)-\frac m{|\Omega|}\bigg)\cdot e^{-(t-t_2^*)}.
\end{equation}
Inserting \eqref{4.10}, \eqref{3.37} and \eqref{3.63b} into \eqref{4.9}, we arrive at
\begin{equation*}
w(\cdot,t)\to\overline{n_0}\mbox{ in }L^\infty(\Omega)
\end{equation*}
exponentially as $t\to\infty$.
\end{proof}

\begin{lemma}\label{lemma4.3*}
Let $m_0$  be as in Lemma \ref{lem3.6} and $\beta\in (\frac 12,1)$. There exist $\alpha>0$ and $C>0$ such that
\begin{equation}\label{4.11}
\|A^\beta u(\cdot,t)\|_{L^2(\Omega)}\leq Ce^{-\alpha t}\mbox{ for all }t>0
\end{equation}
provided that \eqref{1.8} is valid with $\int_\Omega n_0<m_0$.
\end{lemma}
\begin{proof}[Proof] For any $T\geq t_2^*+3$ and $t_2^*<t< T$, we define
$M(t):=e^{\alpha t}\|A^\beta u(\cdot,t)\|_{L^2(\Omega)}$ with $\alpha:=\min\{\mu,\lambda_1,\frac \kappa 6\}$.
Applying the Helmholtz projector $\mathcal{P}$ to the fourth equation in (\ref{3.2}), we have
\begin{equation}\label{4.12}
u_{t}+A u=-\mathcal{P}[(u\cdot\nabla)u]+\mathcal{P}[n\nabla\Phi].
\end{equation}
Now the variation-of-constants formula associated with \eqref{4.12} represents $u$ according to
\begin{equation}\label{4.13}
u(\cdot,t)=e^{-(t-t_2^*)A}u(\cdot,t_2^*)
+\int_{t_2^*}^te^{-(t-s)A}\mathcal{P}[(n-\overline {n_0})\nabla\Phi-(u\cdot\nabla)u](\cdot,s)ds
\end{equation}
for all $t>t_2^*$, where $\mathcal{P} [\overline {n_0}\nabla\Phi]=0$ is used.

From Lemma \ref{3.8}, it follows that there is $C_1>0$ warranting
\begin{equation}\label{4.14}
\|A^\beta u(\cdot,t)\|_{L^2(\Omega)}\leq C_1\leq  C_2 e^{-\mu t}
\end{equation}
with $C_2:= C_1e^ {\mu(t_2^*+1)}$ for all $t_2^*\leq t\leq t_2^* +1$.

Given $\beta\in (\frac 12,1)$, we can choose $r>2$ suitably close to 2 such that $\frac 1 r\in (1-\beta, \frac 12)$
 and then fix $\theta:=\frac r{2(r-1)} \in(0,1)$.
  Using \cite[Lemma 3.1]{Winkler(cvpde2015)} and the Gagliardo--Nirenberg inequality, we obtain $C_3>0$ and $C_4>0$ such that
  $$
  \|\varphi\|_{L^\infty(\Omega)}\leq C_3\|\varphi\|^\theta_{W^{1,r}(\Omega)} \|\varphi\|_{L^2(\Omega)}^{1-\theta}
  \leq C_4\| A^\beta \varphi\|^\theta_{L^{2}(\Omega)} \|\varphi\|_{L^2(\Omega)}^{1-\theta}
  $$
for all $\varphi \in D(A^\beta) \cap W^{1,r}(\Omega)$, and hence
\begin{equation}\label{4.15}
\begin{split}
\|(u\cdot \nabla) u\|_{L^2(\Omega)}\leq &\|u(\cdot,t)\|_{L^\infty(\Omega)}\|\nabla u(\cdot,t)\|_{L^2(\Omega)}\\
\leq &C_4\| A^\beta u\|^\theta_{L^{2}(\Omega)} \|u\|_{L^2(\Omega)}^{1-\theta}\|\nabla u\|_{L^2(\Omega)}.
\end{split}\end{equation}

As for $t>t_2^*+1$, we invoke Lemma \ref{lemma2.2} to estimate the integrals on the right of \eqref{4.13} as follows:
\begin{align}\label{4.16}
&\|A^\beta u(\cdot,t)\|_{L^2(\Omega)}\notag\\
&\leq C_5(t-t_2^*)^{-\beta}e^{-\mu(t-t_2)}
\|u(\cdot,t_2^*)\|_{L^2(\Omega)}
+\int_{t_2^*}^t
\|A^\beta e^{(t-s)A}
\mathcal{P}[(n-\overline {n_0})\nabla\Phi]\|_{L^2(\Omega)}ds\notag\\
&\quad+\int_{t_2^*}^t
\|A^\beta e^{(t-s)A}
\mathcal{P}[(u\cdot\nabla)u]\|_{L^2(\Omega)}ds\\
&\leq C_5(t-t_2^*)^{-\beta}\|u(\cdot,t_2)\|_{L^2(\Omega)} e^{-\mu(t-t_2)}
+C_6\int_{t_2^*}^t(t-s)^{-\beta}
e^{-\mu(t-s)}
\|\mathcal{P}[(n-\overline {n_0})\nabla\Phi]\|_{L^2(\Omega)}ds \notag\\
&\quad+C_6\int_{t_2^*}^t(t-s)^{-\beta}
e^{-\mu(t-s)}
\|\mathcal{P}[(u\cdot\nabla)u]\|_{L^2(\Omega)}ds \notag\\
&\leq C_7 e^{-\mu t}
+C_7\int_{t_2^*}^t(t-s)^{-\beta}
e^{-\mu(t-s)}\|n-\overline{n_0}\|_{L^2(\Omega)}ds +C_6\int_{t_2^*}^t(t-s)^{-\beta}
e^{-\mu(t-s)}\|(u\cdot\nabla)u\|_{L^2(\Omega)}ds. \notag
\end{align}
Here we  use  \eqref{4.8} and Lemma \ref{lemma2.5} in estimating
\begin{equation}\label{4.17}
\int_{t_2^*}^t(t-s)^{-\beta}
e^{-\mu(t-s)}\|n-\overline{n_0}\|_{L^2(\Omega)}ds \leq
 C_8 e^
{-\alpha
t},
\end{equation}
whereas for the last term  in  \eqref{4.16}, we use \eqref{4.15}, \eqref{3.63} and \eqref{3.38} to  estimate
\begin{equation}\label{4.18}
\begin{split}
&\int_{t_2^*}^t(t-s)^{-\beta}
e^{-\mu(t-s)}
\|(u\cdot\nabla)u\|_{L^2(\Omega)}ds \\
\leq
&C_9\int_{t_2^*}^t(t-s)^{-\beta}
e^{-\mu(t-s)}
\| A^\beta u(\cdot,s)\|^\theta_{L^{2}(\Omega)} \|u(\cdot,s)\|_{L^2(\Omega)}^{1-\theta} ds
\\
\leq &
 C_9\sup_{t_2^*<t<T} M^\theta(t) \cdot \int_{t_2^*}^t (t-s)^{-\beta}
e^{-\mu(t-s)} e^{-\alpha s} ds\\
\leq &
 C_{10} e^{-\alpha t}\cdot \sup_{t_2^*<t<T} M^\theta(t).
\end{split}
\end{equation}
Therefore a combination of  \eqref{4.14},  \eqref{4.16}--\eqref{4.18} yields
$$
\|A^\beta u(\cdot,t)\|_{L^2(\Omega)}
\leq  C_{11}  e^{-\alpha t} + C_{11} e^{-\alpha t}   \sup_{t_2^*<t<T} M^\theta(t).
$$
Due to $\theta<1$, we infer the existence of $C_{12}>0$ such that
 $M(t)\leq C_{12}$, which leads to \eqref{4.11} readily.
\end{proof}

We are now in the position to complete the proof of Theorem 1.1.

{\bf Proof of Theorem 1.1.}   \rm
In light of  Lemmas \ref{lemma4.1},  \ref{lemma4.3} and \ref{lemma4.3*}, to establish the exponential stabilization of the solution to \eqref{1.6}, it suffices to show
the decay property of  $v$ and $\|\nabla z(\cdot,t)\|_{L^\infty(\Omega)}$.
From the decay property of $w-m$ in \eqref{4.4}, it follows that one can find $t_1>0$ sufficiently large such that for all $x \in \Omega $  and $t\geq t_1$,
 $w(x,t)>\frac {3m} 4 $ and thereupon
$$
v_t\leq \triangle v-u\cdot\nabla v-\frac {3m} 4 v.
 $$
 Hence, by means of a straightforward comparison with the ODE
$y'(t)=-\frac {3m} 4 y$ where $y(t_1)=\|v_{0}\|_{L^{\infty}(\Omega)}$,
 we see that
 \begin{equation}\label{4.19}
	\begin{split}	
 v(x,t)\leq  \|v_{0}\|_{L^{\infty}(\Omega)} e^{-\frac {3m} 4 (t-t_1)}.
	\end{split}
\end{equation}
%and thereby completes the proof of this theorem.

In addition, thanks to \eqref{3.7} and  \eqref{3.64},   the Gagliardo--Nirenberg inequality provides $C_1>0$ with
$$
\|z(\cdot, t)\|_{L^2(\Omega)}\leq C_{GN}\|\nabla z(\cdot,t)\|^{\frac 25}_{L^4(\Omega)}\| z(\cdot,t)\|^{\frac 35}_{L^1(\Omega)(\Omega)}+C_{GN}
\| z(\cdot,t)\|_{L^1(\Omega)}\leq C_1(t+1).
$$
On the  other hand, we also have
\begin{equation}\label{4.9}
\|w(\cdot,t)-\overline{w}(t)\|_{L^6(\Omega)}
\leq \|\nabla w(\cdot,t)\|_{L^2(\Omega)}^{\frac23}
\|w(\cdot,t)-\overline{w}(t)\|_{L^2(\Omega)}^{\frac13}.
\end{equation}
Therefore similarly reasoning as in \eqref{4.2},  applying \eqref{4.1} to $q=6$, Lemmas \ref{lemma4.3} and \ref{lemma2.5}, we can conclude that
there exist $C_i>0\, (i=2,3,4)$ such that
%which along with Lemma \ref{lemma2.1} and Lemma \ref{lemma3.1} shows that for all $t>T_1+2$
	\begin{equation*}\label{5.11}
		\begin{split}
			&\|\nabla z(\cdot,t)\|_{L^\infty(\Omega)}\\
\leq &C_2(1+(t-t_2^*)^{-1}) e^{-\lambda_1(t-t_2^*)} \|z(\cdot, t_2^*)\|_{L^2(\Omega)}
+
\int_{t_2^*}^{t}
\| \nabla e^{(t-s)\Delta}|\nabla z|^2(\cdot,s)\|_{L^\infty(\Omega)}ds\\
&+\int_{t_2^*}^{t}
\|\nabla e^{(t-s)\Delta} (w(\cdot,s)-\overline w(s))\|_{L^\infty(\Omega)}ds
+\int_{t_2^*}^t\|\nabla e^{(t-s)\Delta} (u\cdot\nabla z)\|_{L^\infty(\Omega)}ds\\
\leq &2 C_2 e^{-\lambda_1(t-t_2^*)}
+C_2\int_{t_2^*}^{t}(1+(t-s)^{-\frac 23}) e^{-\lambda_1(t-s)}\|\nabla z(\cdot,s)\|^2_{L^6(\Omega)}ds\\
			 &	
+ C_2\int_{t_2}^{t}(1+(t-s)^{-\frac 23})e^{-\lambda_1(t-s)}
\|w(\cdot,s)-\overline w(s)\|_{L^6(\Omega)}ds\\
&+ C_2\int_{t_2^*}^{t}(1+(t-s)^{-\frac 23})e^{-\lambda_1(t-s)}
\|u(\cdot,s)\|_{L^\infty(\Omega)}\|\nabla z(\cdot,s)\|_{L^6(\Omega)}
ds\\
		\leq	&	
2 C_2 e^{-\lambda_1(t-t_2^*)}
+C_3
\int_{t_2^*}^{t}(1+(t-s)^{-\frac 23}) e^{-\lambda_1(t-s)} e^{-2\kappa s}ds\\ % \sigma \sup_{t\geq T_1+1}\|\nabla z(\cdot,t)\|^2_{L^6(\Omega)}\\
&+C_3
\int_{t_2^*}^{t}(1+(t-s)^{-\frac 2 3}) e^{-\lambda_1(t-s)} e^{-\frac{\kappa} 6 s}ds\\
\leq& C_4
 e^{-\min\{
 \lambda_1,\frac \kappa 6\} t}
 		\end{split}
	\end{equation*}
for all $t>t_2^*+2$, where Lemma 2.4 is used.
% According to Lemma 6.3 of \cite{Winkler(2014)}, we can establish the uniform decay of $u$,
%and satisfy $\|u(\cdot,t)\|_{L^\infty(\Omega)}\to0$ as $t\to\infty$.
This completes the proof of the main result of this paper.

\vspace{12pt}

{\bf Acknowledgments}:
This work was partially supported by the National Natural Science Foundation of China (no.~12071030; third author) and NUS Academic Research Fund (no.~A-0004282-00-00; second and third authors). The second author also acknowledges the hospitality of the Fields Institute for Research in Mathematical Sciences in Toronto where part of this research was carried out.

\end{document}